\documentclass{article}
\usepackage{amsmath,amsthm,amssymb,amscd,bm}
\usepackage{ascmac}

\usepackage[arrow,matrix]{xy}

\numberwithin{equation}{section}
\numberwithin{figure}{section}

\newtheorem{thm}{Theorem}[section]
\newtheorem*{thm*}{Theorem}
\newtheorem*{lem*}{Lemma}
\newtheorem{lem}[thm]{Lemma}
\newtheorem*{cor*}{Corollary}
\newtheorem{cor}[thm]{Corollary}
\newtheorem*{proposition*}{Proposition}
\newtheorem{proposition}[thm]{Proposition}
\theoremstyle{definition}
\newtheorem*{definition*}{Definition}

\newtheorem*{rem*}{Remark}
\newtheorem{rem}[thm]{Remark}
\newtheorem{cdt}[thm]{Condition}

\def\e#1\e{\begin{equation}#1\end{equation}}
\def\iz#1\iz{\begin{itemize}#1\end{itemize}}
\def\ea#1\ea{\begin{align}#1\end{align}}
\def\eq{\eqref}
\def\l{\label}
\def\0{\hspace{0pt}}
\def\sq{\sqrt}
\def\cyl{\rm{cyl}}

\def\Xh{\widehat X}

\def\Kh{\widehat K}

\def\etati{\widetilde{\eta}}

\def\lg{\langle}
\def\rg{\rangle}

\def\dim{\mathop{\rm dim}\nolimits}

\def\Ker{\mathop{\rm Ker}}

\def\ge{\geqslant}
\def\le{\leqslant\nobreak}

\def\cC{{\mathbin{\cal C}}}

\def\cE{{\mathbin{\cal E}}}

\def\cH{{\mathbin{\cal H}}}

\def\cM{{\mathbin{\cal M}}}

\def\cN{{\mathbin{\cal N}}}
\def\oM{{\mathbin{\smash{\,\,\ov{\!\!\mathcal
M\!}\,}}}}

\def\cT{{\mathbin{\cal T}}}
\def\cU{{\mathbin{\cal U}}}
\def\cV{{\mathbin{\cal V}}}
\def\={\equiv}

\def\cX{{\mathbin{\cal X}}}
\def\cY{{\mathbin{\cal Y}}}

\def\C{{\mathbin{\mathbb C}}}

\def\R{{\mathbin{\mathbb R}}}

\def\al{\alpha}
\def\be{\beta}
\def\ga{\gamma}
\def\de{\delta}

\def\ep{\epsilon}

\def\th{\theta}

\def\si{\sigma}
\def\om{\omega}

\def\Si{\Sigma}

\def\Th{\Theta}
\def\Om{\Omega}
\def\Ga{\Gamma}

\def\d{\partial}

\def\ts{\textstyle}

\def\-{\setminus}
\def\bu{\bullet}
\def\op{\oplus}

\def\ov{\overline}

\def\iy{\infty}

\def\t{\times}
\def\ci{\circ}

\def\ha{{\frac{1}{2}}}

\def\su{\subset}

\DeclareMathOperator{\grad}{grad}

\DeclareMathOperator{\gr}{Graph}
\DeclareMathOperator{\di}{div}
\DeclareMathOperator{\area}{area}
\DeclareMathOperator{\re}{Re}
\DeclareMathOperator{\im}{Im}
\DeclareMathOperator{\spt}{Spt}

\title{Surjectivity of a Gluing Construction in Special Lagrangian Geometry}
\author{Yohsuke Imagi}
\date{Kavli IPMU (WPI)}
\begin{document}
\maketitle
\textbf{Abstract.}
This paper is motivated by a relatively recent work by Joyce \cite{J1,J2,J3,J4,J5} in special Lagrangian geometry, but the basic idea of the present paper goes back to an earlier pioneering work of Donaldson \cite{D} (explained also by Freed and Uhlenbeck \cite{FU}) in Yang--Mills gauge theory; Donaldson discovered a global structure of a (compactified) moduli space of Yang--Mills instantons, and a key step to that result was the proof of surjectivity of Taubes' gluing construction \cite{Taubes}.

In special Lagrangian geometry we have currently no such a global understanding of (compactified) moduli spaces, but in the present paper we determine a neighbourhood of a `boundary' point.
It is locally similar to Donaldson's result, and in particular as Donaldson's result implies the surjectivity of Taubes' gluing construction so our result implies the surjectivity of Joyce's gluing construction in a certain simple case.

\section{Introduction}\l{intro}
The main result of the present paper may be stated briefly as the surjectivity of Joyce's gluing construction in a certain simple case, and we begin therefore with a review of Joyce's work \cite{J1,J2,J3,J4,J5}.

Let $M$ be a Calabi--Yau (or more generally almost\footnote{
This is a terminology of Joyce \cite{J1,J2,J3,J4,J5} which means that $M$ need {\it not} be Ricci-flat.
}Calabi--Yau) manifold of complex dimension $m$,
and let $X$ be a compact special Lagrangian $m$-fold in $M$ with finitely many singular points $x,y,\cdots,z\in X$ modelled on multiplicity-one special Lagrangian cones $C_x,C_y,\cdots,C_z\su\C^m$ with isolated singularity.

Joyce \cite{J3,J4} studied a {\it smoothing}\footnote{
Joyce \cite{J3,J4} calls it a {\it desingularization} but in algebraic geometry it means a {\it resolution} of singularity which does not fit our context, and we shall therefore call it a smoothing.
}of $X$ by the gluing technique, which may be sketched as follows:
as local smoothing models for $C_x,C_y,\cdots,C_z$ let $L_x,L_y,\cdots,L_z$ be non-compact special Lagrangian submanifolds properly-{\0}embedded in $\C^m$ and asymptotic at infinity to $C_x,C_y,\cdots,C_z$ with multiplicity $1$ respectively;
for each $t>0$ let $tL_x:=\{t\bm z\in\C^m:\bm z\in L_x\}$ and define $tL_y,\cdots,tL_z$ likewise;
then under some hypotheses we can glue $tL_x,tL_y,\cdots,tL_z$ to $X$ at $x,y,\cdots,z$ respectively into a family of compact special Lagrangian submanifolds $N_t$ of $M$ with small $t>0$ and tending to $X$ as $t\to+0$ as {\it varifolds}\footnote{
Another well-known notion in geometric measure theory are {\it currents} and there is some difference between varifolds and currents, but the difference will not matter in special Lagrangian geometry (or more generally {\it calibrated} geometry) as we shall explain in \S\ref{SL section} below.
}, which are a notion of `singular' submanifolds in geometric measure theory.

Consider now the space $\cV$ consisting of all compactly-supported special Lagrangian integral varifolds with no boundary in $M$, so that $X\in\cV$,
where $X$ has isolated singular points and multiplicity-one tangent cones but general elements of $\cV$ may have {\it non-isolated} singularity and {\it higher-multiplicity} tangent cones.

We wish to detemine a {\it neighbourhood} of $X$ in $\cV$. For some simple $X$ indeed we can determine a neighbourhood of $X$ in $\cV$ in a way similar to Donaldson's work (explained also by Freed and Uhlenbeck \cite{FU}) in Yang--Mills gauge theory.

Donaldson compactifies a moduli space $\cM$ of Yang--Mills instantons by adding some objects with {\it isolated} singularity (by Uhlenbeck's theorem \cite{Uh}) and those singular objects form the boundary $\d\cM$ of the compactified moduli space $\oM=\cM\cup\d\cM$.

Donaldson determines indeed a neighbourhood $\cN$ of $\d\cM$ in $\oM$ and its proof is based on the bubbling-off (or blowing-up) analysis using a technique of Uhlenbeck \cite{Uh} and on the classification of local models by Atiyah, Hitchin and Singer~\cite{AHS}.

In our situation the basic tool necessary for the blowing-up analysis is already prepared in the preceding paper \cite{Imagi}, which will be explained in \S\ref{bu} below. On the other hand the classification of local models $L_x,L_y,\cdots,L_z$ will be difficult in general.
Consider therefore the simple case where $m=3$ and the cones $C_x,C_y,\cdots,C_z$ are all equal to
\begin{equation}\l{C}
C:=\{(z_1,z_2,z_3)\in\C^3\-\{0\}
:|z_1|=|z_2|=|z_3|,z_1z_2z_3\in(0,\iy)\}
\end{equation}
which is a special Lagrangian cone in $\C^3$ discovered by Harvey and Lawson~\cite[Chapter~III.3.A, Theorem~3.1]{HL}. It is also {\it stable} in the sense of Joyce \cite[\S3.2]{J2}.
In the unit sphere $S^5:=\{(z_1,z_2,z_3)\in\C^3:|z_1|^2+|z_2|^2+|z_3|^2=1\}$
it is easy to see that $C\cap S^5$ is diffeomorphic to $T^2$ and so $C$ is a $T^2$-cone.
Haskins~\cite{Haskins} proves that $C$ is, as a stable $T^2$-cone, unique up to $SU_3$-rotation.

To state our classification theorem we also define:
\ea
\l{L1}L_1:=\{(z_1,z_2,z_3)\in\C^3:|z_1|^2-1=|z_2|^2=|z_3|^2,
z_1z_2z_3\in[0,\iy)\},\\
L_2:=\{(z_1,z_2,z_3)\in\C^3:|z_1|^2=|z_2|^2-1=|z_3|^2,
z_1z_2z_3\in[0,\iy)\},\\
\l{L3}L_3:=\{(z_1,z_2,z_3)\in\C^3:|z_1|^2=|z_2|^2=|z_3|^2-1,
z_1z_2z_3\in[0,\iy)\},
\ea
which are all non-compact special Lagrangian submanifolds properly-embedded in $\C^3$ and asymptotic at infinity to $C$ with multiplicity $1$.
In particular $L_1,L_2,L_3$ are all {\it non-singular}, despite of the condition $z_1z_2z_3\in[0,\iy)$;
for instance the map $S^1\t\C\to L_1\su\C^3$ given by $(e^{i\th},z)\mapsto(e^{i\th}\sq{|z|^2+1},z,e^{-i\th}\bar z)$ is a diffeomorphism;
permuting the co-ordinates of $\C^3$ we also get diffeomorphisms onto $L_2$ and $L_3$.

Our classification theorem is:
\begin{thm}\l{uniq}
Let $W$ be a special Lagrangian varifold in $\C^3$ with no boundary asymptotic at infinity to $C$ with multiplicity $1$.
Then $W$ is a mutiplicity-one varifold represented by $C$ or $sL+b$ for some $L\in\{L_1,L_2,L_3\},$ $t>0$ and $b\in\C^3$ where $sL+b:=\{t\bm z+b\in\C^3:\bm z\in L\}$.
\end{thm}

The proof of Theorem \ref{uniq} will be given in \S\ref{un} and at the moment we only point out that the proof uses a {\it symmetry} of $C$.

%For the proof we shall use the stability and a symmetry of $C$.
%Our method seems to generalize to stable cones with a suitable symmetry,
%such as examples of Ohnita~\cite{Ohnita}.
%For instance, one can embed $SU_3/SO_3$ into $S^{11}$,
%and prove that the cone on $SU_3/SO_3$ is a special Lagrangian cone in $\C^6$,
%which is stable by a result of Ohnita~\cite{Ohnita}.
%Haskins and Pacini~\cite[Remark~2.3]{Haskins Pacini}
%point out that $SU_3/SO_3$ 
%has a non-zero Stiefel--Whitney number,
%and so does not bound any compact manifold of dimension $6$.
%This implies that one cannot smooth
%compact special Lagrangian $6$-folds with singularities modelled on
%the cone on $SU_3/SO_3$ of multiplicity $1$.
%If one has an analogue to Theorem~\ref{uniqueness} for
%the cone on $SU_3/SO_3$, then
%using our bubbling-off analysis one can detemine
%neighbourhoods in $\cV$
%of compact special Lagrangian $6$-folds with singularities modelled on
%the cone on $SU_3/SO_3$ of multiplicity $1$.

We return now to the special Lagrangian geometry in the (almost) Calabi--Yau manifold $M$.
Let $X$ be a compact special Lagrangian $3$-fold $X$ with {\it only one} singular point $x$ modelled on $C$ so that we can apply Theorem \ref{uniq} above.
One may also consider of course two or more singular points in which case however one has to consider their {\it interaction} as Joyce does \cite[\S10.3]{J5}---we shall not discuss it in the present paper.

To state our main results we also introduce the subspace $\cX\su\cV$ consisting of those $Y$ which are singular only at one point and modelled on $C$ with multiplicity $1$, so that $X\in\cX$.
Our main results may then be summarized briefly into the following single statement:
\begin{thm}\l{main1}
There exists a neighbourhood $\cU$ of $X$ in $\cV$ such that any element of $\cU\-\cX$ may be obtained by Joyce's gluing construction.
\end{thm}

Here $\cU\-\cX$ may be {\it empty}, in which case we have $\cU\su\cX$ and so $X$ is {\it unsmoothable}.

We shall state a more precise meaning of `Joyce's gluing construction' in Theorem \ref{main1}.
As in \eq{L1}--\eq{L3} we have three local models $L_1,L_2,L_3$ and so there are apparently three distinct ways of smoothing $X$ but in fact there is {\it at most one}\footnote{
There will be possibly more than one way of smoothing $X$ if one considers a family of (almost) Calabi--Yau manifolds as Joyce does \cite[\S10.2]{J5}.
}way, which may be explained as follows.

For any $L\in\{L_1,L_2,L_3\}$ we can certainly glue $L$ to $X$ at $x$ into a compact submanifold of $M$
but to make it {\it Lagrangian} in $M$ we need a topological condition between $L$ and $X$ given by Joyce \cite[Theorem 10.4 (see also Theorem 7.3)]{J5}.
There may be no $L\in\{L_1,L_2,L_3\}$ for which $L$ and $X$ satisfy the topological condition, which will be recalled in \S\ref{cd} below.
Joyce \cite[Proposition 10.3]{J1} proves indeed:
\begin{lem}\l{j}
{\rm
There is {\it at most one} $L\in\{L_1,L_2,L_3\}$ for which $L$ and $X$ satisfy the topological condition.
}
\end{lem}

Theorem \ref{main1} includes the following statement:
\begin{thm}\l{main2}
If there is no $L\in\{L_1,L_2,L_3\}$ for which $L$ and $X$ satisfy the topological condition of Joyce {\rm\cite[Theorem 10.4]{J5}} then there exists a neighbourhood of $X$ in $\cV$ contained in $\cX$.
\end{thm}

It is easy to see that that if there is $L\in\{L_1,L_2,L_3\}$ for which $L$ and $X$ satisfy the topological condition, then so do $L$ and $Y$, where $Y$ is not exactly $X$ but in a {\it neighbourhood} $\cY$ of $X$ in $\cX$.
Consequently $L$ and $Y$ may be glued together into a compact special Lagrangian submanifold of $M$.

More precisely we re-scale $L$ by small $t>0$ and glue $tL$ to $X$ at $x$.
Making $\cY$ smaller if necessary we can indeed find a real number $\tau>0$ and define a continuous map $G:[0,\tau)\t\cY\to\cV$ with the following two properties:
\iz
\item[(i)] $G(0,Y)=Y$ for all $Y\in\cY$;
\item[(ii)] if $(t,Y)\in(0,\tau)\t\cY$ then $G(t,Y)$ is a compact special Lagrangian submanifold of $M$ obtained by gluing $Y$ and $tL$ together.
\iz
Theorem \ref{main1} may be then refined as follows:
\begin{thm}\l{main3}
{\rm
If there is $L\in\{L_1,L_2,L_3\}$ for which $L$ and $X$ satisfy the topological condition of Joyce {\rm\cite[Theorem 10.4]{J5}} then $G:[0,\tau)\t\cY\to\cV$ is a {\it homeomorphism onto a neighbourhood} $\cU$ of $X$ in $\cV$ with $\cU\cap\cX=\cY$.
}
\end{thm}

Thus $\cU$ is a {\it collar} neighbourhood of $\cY$ in $\cV$ and so locally similar to Donaldson's situation \cite{D} which also contains a collar neighbourhood of the boundary $\d\cM$ of the compactified moduli space $\oM$ in the notation above.

We also note that $\cY$ and $\cU\-\cY$ are {\it manifolds} of finite dimension, and that $G$ maps $(0,\tau)\t\cY$ {\it diffeomorphically} onto $\cU\-\cY$.
These facts may be proven in the following three steps:
\iz
\item[(i)]
As $\cU\-\cY$ consists of compact special Lagrangian (non-singular) submanifolds of $M$ it follows from McLean's theorem \cite[Theorem 3.6]{M} that $\cU\-\cY$ is a manifold but with respect to the $C^\iy$-topology.
By Allard's regularity theorem \cite[Theorem 8.19]{Allard} however the $C^\iy$-topology on $\cU\-\cY$ is equal to the varifold topology induced from $\cV$.
\item[(ii)]
Joyce \cite[Corollary 6.11]{J2} extended McLean's theorem to the {\it stable-cone} singularity and as $C$ is stable (see \cite[\S3.2]{J2}) we can apply Joyce's theorem so that $\cY$ will be a manifold with respect to a strong topology (given by Joyce \cite[Definition 5.6]{J2}). It is again equal to the varifold topology, which we prove in Theorem \ref{same} below.
\item[(iii)]
The fact that $G$ maps $(0,\tau)\t\cY$ diffeomorphically onto $\cU\-\cY$ is already observed by Joyce (see a discussion after \cite[Definition 8.9]{J5});
there are in fact natural co-ordinate systems on $(0,\tau)\t\cY$ and $\cU\-\cY$ with respect to which $G$ is a {\it product} map.
\iz

Thus Theorems \ref{main2} and \ref{main3} are the precise meaning of Theorem \ref{main1}.
Their proof will be given in \S\S\ref{bg1}--\ref{pr3}, the last three sections of the present paper.
It will be again similar to the corresponding part of Donaldson's proof \cite[III.4]{D} (see also Freed and Uhlenbeck \cite[\S9, Connectivity of the Collar]{FU}).

We point out that Theorem \ref{main3} implies the {\it connectivity} of a neighbourhood of $X$ in $\cV$, which part requires a careful treatment in particular. It is also the case in Donaldson's situation.

The remaining sections may be summarized as follows.
In \S\ref{SL section} we give a review of geometric measure theory and {\it calibrated} geometry including special Lagrangian geometry.
There are two notions of `singular' submanifolds in geometric measure theory,
varifolds and currents, and there is some difference between them in general, but it does not matter in calibrated geometry as we shall explain in \S\ref{SL section}.

We shall mainly use varifolds rather than currents so that we can directly use Allard's regularity theorem \cite[Theorem 8.19]{Allard} which for instance we have used above in the proof that the two topologies on $\cU\-\cY$, the $C^\iy$-one and the varifold one, are the same.

In \S\ref{bu} we analyse the blowing-up of special Lagrangian varifolds in a general situation, concerning a multiplicity-one special Lagrangian {\it Jacobi-integrable} cone with isolated singularity, which fits Joyce's framework \cite[Definition 6.7]{J5}.

In \S\ref{un} we prove Theorem \ref{uniq} as we have mentioned above.

Some material of \S\S\ref{SL section}--\ref{un} is not directly relevant to the main results (Theorems \ref{main2} and \ref{main3}) but may be of independent interest and of potential use in more general situations.

\subsection*{Acknowledgements}
I would like to thank Kenji Fukaya,
Manabu Akaho, Mark Haskins,
Lorenzo Foscolo, Yoshihiro Tonegawa
and the referees
for many useful comments.
I was supported in part by a JSPS research fellowship 22-699
and a GCOE programme at Kyoto University.
Part of this paper was written during my visit to Imperial College London.

\section{Calibrated Geometry and Geometric Measure Theory}
\l{SL section}
In this section we shall give a review of
special Lagrangian geometry (a kind of calibrated geometry) and geometric measure theory.
First of all we shall define almost Calabi--Yau manifolds and
their special Lagrangian submanifolds.
%Harvey and Lawson
%discuss special Lagrangian geometry in Calabi--Yau manifolds,
%but we can define special Lagrangian geometry also in almost
%Calabi--Yau manifolds.
%Goldstein~\cite{Goldstein} and others:
%Goldstein constructs some examples of almost Calabi--Yau manifolds
%and their special Lagrangian submanifolds;
%Joyce~\cite{JoyceSYZ} discusses generic
%special Lagrangian submanifolds of almost Calabi--Yau manifolds.

Let $(M,\om)$ be a symplectic manifold of dimension $2m$,
and let $J$ be a complex structure on $M$
such that if we put
$\hat{g}(v,w)=\om(v,Jw)$
then $\hat{g}$ will be a K\"ahler metric on $(M,J)$.
Let $\Om$ be a holomorphic $(m,0)$-form on $(M,J)$
with $\Om|_x\neq0$ for every $x\in M$.
Then we shall call $(M,\om,J,\Om)$ an almost Calabi--Yau manifold,
and $(\om,J,\Om)$ an almost Calabi--Yau structure on $M$.
We can define a smooth function $\psi:M\to(0,\iy)$ such that
\begin{equation}\l{psi}
\frac{\psi^{2m}}{m!}\om^{\wedge m}=
(-1)^s\left(\frac{i}{2}\right)^m
\Om\wedge\ov{\Om}
\text{ where }s=\frac{m(m-1)}{2}.
\end{equation}
We put $g=\psi^2\hat{g}$.
We shall call $g$ the almost Calabi--Yau metric on $(M,\om,J,\Om)$.
Here $g$ need not be a K\"ahler metric on $(M,J)$.
If we have $\psi(x)=1$ for every $x\in M$ then
$g$ will be a K\"ahler metric of Ricci curvature $0$.
In that case we shall call $(\om,J,\Om)$ a Calabi--Yau structure on $M$,
and $g$ the Calabi--Yau metric on $(M,\om,J,\Om)$.

We define a Calabi--Yau structure on $\R^{2m}=\C^m$ as follows.
Let $(z^1,\cdots,z^m)$ be the co-ordinates of $\C^m$.
Let
$\om_0=\frac{i}{2}(dz^1\wedge d\ov{z^1}
+\cdots+dz^m\wedge d\ov{z^m})$,
$g_0=dz^1\otimes d\ov{z^1}+
\cdots+dz^m\otimes d\ov{z^m}$
and $\Om_0=dz^1\wedge\cdots\wedge dz^m$.
Let $J_0$ be
the complex structure of $\C^m$.
Then $(\om_0,J_0,\Om_0)$
is a Calabi--Yau structure on $\R^{2m}$,
and $g_0$ is the almost Calabi--Yau metric on $(\R^{2m},\om_0,J_0,\Om_0)$.

Let $(M,\om,J,\Om)$ be an almost Calabi--Yau manifold,
and let $g$ be the almost Calabi--Yau metric on $(M,\om,J,\Om)$.
Then $\re\Om$ will be a calibration of degree $m$ on $(M,g)$
in the sense of Harvey and Lawson~\cite{HL};
i.e. for any point $x\in M$ and any $\R$-linear subspace $S\su T_xM$ with $\dim_\R S=m$ we have $|(\re\Om)_x|_S|\le1$ where the norm is induced from the metric $g$ on $M$.

Special Lagrangian submanifolds of $(M,\om,J,\Om)$ are defined as real $m$-dimensional submanifolds $N$ of $M$ with $|\re\Om|_N|=1$ where the norm is again induced from the metric $g$ on $M$.
Special Lagrangian submanifolds of $(M,\om,J,\Om)$
will be Lagrangian with respect to $\om$ and area-minimizing with respect to $g$.

We can also define $\re\Om$-varifolds and currents in $(M,g)$,
which we shall call special Lagrangian varifolds and currents in
$(M,\om,J,\Om)$, respectively.
In the remainder of this section we shall consider
calibrated geometry not limited to special Lagrangian geometry.

We suppose that $M$ is a manifold.
For each $x\in M$
we denote by $G_p(T_xM)$
the Grassmann manifold of all vector subspaces of $T_xM$ of dimension $p$.
We put $G_p(TM)=\bigcup_{x\in M}G_p(T_xM)$.
By a varifold of dimension $p$ in $M$
we shall mean a Radon measure on $G_p(TM)$.

We suppose that $g$ is a Riemannian metric on $M$,
and $\phi$ is a calibration of degree $p$ on $M$.
For each $x\in M$ we put
$G_\phi(T_xM)=\{S\in G_p(T_xM):|\phi|_S|_g=1\}$.
We also put $G_\phi(TM)=\bigcup_{x\in M}G_\phi(T_xM)$.
By a $\phi$-varifold in $M$
we shall mean a Radon measure on $G_\phi(TM)$.

For each $x\in M$ and $S\in G_{\phi}(T_xM)$
we define $\overrightarrow{S}\in\bigwedge^p T_xM$ as follows.
We take an orthonormal basis $(e_1,\cdots,e_p)$ for $T_xM$ with respect to $g|_x$
such that $\lg\phi|_x,e_1\wedge\cdots\wedge e_p\rg=1$.
We set
$\overrightarrow{S}=e_1\wedge\cdots\wedge e_p$.
It is easy to see
that $\overrightarrow{S}$ is independent of the choice of $(e_1,\cdots,e_p)$,
and so well-defined.

Let $V$ be a $\phi$-varifold in $M$.
Then we can define a $p$-current $\overrightarrow{V}$ in $M$ by setting
\[\overrightarrow{V}(\chi)=\int_{G_{\phi}(TM)}
\lg\chi|_x,\overrightarrow{S}\rg dV(x,S)\]
for every compactly-supported $p$-form $\chi$ on $M$.

Harvey and Lawson~\cite[Chapter~II.1, Definition~1.4]{HL}
define positive $\phi$-currents in $M$,
which we shall explain next.
First of all we recall a definition of
Harvey and Lawson~\cite[Chapter~II.A, Definition~A.1]{HL}:
by a $\phi$-non-negative $p$-form on $M$
we shall mean a $p$-form $\chi$ on $M$
with $\lg\chi|_x,\overrightarrow{S}\rg\geq0$ for every
$x\in M$ and $S\in G_\phi(T_xM)$.
Harvey and Lawson~\cite[Chapter~II.A, Proposition~A.2]{HL}
prove that a $p$-current $T$ in $M$
is a positive $\phi$-current
if (and only if) we have $T(\chi)\geq0$
for every compactly-supported $\phi$-non-negative $p$-form on $M$.
We have:
\begin{thm*}
Let $V$ be a $\phi$-varifold in $M$.
Then $\overrightarrow{V}$ is a positive $\phi$-current in $M$
in the sense of Harvey and Lawson.
\end{thm*}
\begin{proof}
We have only to prove that
if $\chi$ is a compactly-supported $\phi$-non-negative $p$-form on $M$
then we have $\overrightarrow{V}(\chi)\geq0$.
By the definition of $\phi$-non-negative $p$-forms we have
$\lg\chi|_x,\overrightarrow{S}\rg\geq0$
for every $x\in M$ and $S\in G_\phi(T_xM)$.
Hence we get
$\overrightarrow{V}(\chi)=\int
\lg\chi|_x,\overrightarrow{S}\rg dV(x,S)\geq0$,
completing the proof.
\end{proof}

We suppose now that $M$ is compact.
We take $a\in H_p(M;\R)$.
We denote by $\cV$
the set of all compactly-supported integral $\phi$-varifolds $V$ in $M$,
$\d\overrightarrow{V}=0$ and $[\overrightarrow{V}]=a$.
Here $V$ integral means that there exists an integer-valued $\|V\|$-measurable function $\Th_V$ (called the {\it multiplicity} function of $V$) such that for any compactly-supported continuous function $f:M\to[0,\iy)$ we have
$\|V\|(f)=\int_Mf\Th_V\cH^p$ where $\cH^p$ denotes the $p$-dimensional Hausdorff measure in $(M,g)$.

We give $\cV$ the weak topology in the sense of Allard
\cite[Definition~2.6(2)]{Allard},
i.e. the topology of the Radon measures on $G_p(TM)$.
We have then:
\begin{thm*}
$\cV$ is compact.
\end{thm*}
\begin{proof}
It is easy to see that $\cV$ is a metrizable space.
It suffices therefore to prove that if $V_1,V_2,V_3,\cdots\in\cV$
then there exists a subsequence of $(V_n)_{n=1}^\iy$
converging in $\cV$.
Since $M$ is compact it is clear that $\phi$ is compactly-supported,
and for each $n=1,2,3,\cdots$ therefore we have
$\area V_n=\overrightarrow{V_n}(\phi)=a\cdot[\phi]$
where $[\phi]$ denotes the de Rham cohomology class of $\phi$.
This implies that $\area V_n$ is bounded with respect to $n$.
By an integral compactness theorem of Allard~\cite[Theorem~6.4]{Allard},
therefore, we can find a subsequence of $(V_n)_{n=1}^\iy$
converging as Radon measures on $G_p(TM)$.
We may identify $(V_n)_{n=1}^\iy$
with the subsequence.
We denote its limit by $V$.
It suffices then to prove $V\in\cV$.
Allard's \cite[Theorem~6.4]{Allard} integral compactness theorem implies that $V$ is an integral varifold.
Since $V_n$ tends to $V$ as varifolds
we see that $\overrightarrow{V_n}$ tends to $\overrightarrow{V}$
as $p$-currents in $M$.
As $\d\overrightarrow{V_n}=0$
for every $n=1,2,3,\cdots$ so we have $\d\overrightarrow{V}=0$.
As $[\overrightarrow{V_n}]=a$
for every $n=1,2,3,\cdots$ so we have $[\overrightarrow{V}]=a$.
We have therefore $V\in\cV$ as we want.
\end{proof}

Let $V$ be a varifold of dimension $p$ in $M$.
Then we shall denote by $\|V\|$ the Radon measure on $M$ defined by
setting $\|V\|(f)=V(f\ci\pi)$ for every $f\in C_c(M;\R)$,
where $\pi$ denotes the projection of $G_p(TM)$ onto $M$,
and $C_c(M;\R)$ denotes the set of all compactly-supported continuous functions on $M$.

We denote by $\mathcal{R}$ the set of all Radon measures on $M$
which may be expressed as $\|V\|$ for some $V\in\cV$.
We give $\mathcal{R}$ the topology of the Radon measures on $M$.
We have then:
\begin{thm*}
The mapping $V\mapsto\|V\|$
is a homeomorphism of $\cV$ onto $\mathcal{R}$.
\end{thm*}
\begin{proof}
From the definition of $\mathcal{R}$ it is clear that
$V\mapsto\|V\|$ maps $\cV$ onto $\mathcal{R}$.
We claim that the mapping $V\mapsto\|V\|$ is one-to-one.
As $V$ is area-minimizing it follows from Allard's theorem \cite[Theorem 5.5] {Allard} that $V$ is {\it rectifiable} so that for each $V\in\cV$ and $f\in C_c(M;\R)$ we have
\e\l{Vf}V(f)=\int_{x\in M}f(x,T_x\|V\|)d\|V\|(x)\e
where $T_x\|V\|$ denotes the (approximate) tangent space to $\|V\|$ at $x$, which exists for $\|V\|$-almost every $x\in M$
(the terminology `approximate' tangent space is used by Simon \cite[see Remark 38.2]{Simon3}).
The expression \eq{Vf} shows that $V$ is determined by $\|V\|$ so that the mapping $V\mapsto\|V\|$ is one-to-one.

From the definition of $\|V\|$, moreover, we see that
$V\mapsto\|V\|$ is continuous.
We have thus proved that $V\mapsto\|V\|$
is a continuous bijection of $\cV$ onto $\mathcal{R}$.
Notice also that $\cV$ is compact and $\mathcal{R}$ is Hausdorff.
Then we see that
$V\mapsto\|V\|$
is a homeomorphism of $\cV$ onto $\mathcal{R}$.
\end{proof}

We denote by $\cT$ the set of all positive $\phi$-currents in $M$
which may be expressed as $\overrightarrow{V}$ for some $V\in\cV$.
We give $\cT$ the topology of $p$-currents in $M$.
We have then:
\begin{thm*}
The mapping $V\mapsto\overrightarrow{V}$
is a homeomorphism of $\cV$ onto $\cT$.
\end{thm*}
\begin{proof}
From the definition of $\cT$ it is clear that
$V\mapsto\overrightarrow{V}$ maps $\cV$ onto $\cT$.
We claim that the mapping $V\mapsto\overrightarrow{V}$ is one-to-one.
For each $V\in\cV$ and $f\in C_c(M;\R)$ we have
\[\overrightarrow{V}(f\phi)=\int_{(x,S)\in G_\phi(TM)}\lg f(x)\phi|_x,S\rg dV(x,S)=\int_{G_\phi(TM)}f(x)dV=\|V\|(f).\]
Thus $\overrightarrow V$ determines $\|V\|$.
On the other hand we have already seen in \eq{Vf} that $\|V\|$ determines $V$.
Consequently $\overrightarrow V$ determines $V$; more precisely the mapping $V\mapsto\overrightarrow V$ is one-to-one.

It is easy to see that $V\mapsto\overrightarrow{V}$ is continuous.
We have thus proved that $V\mapsto\overrightarrow{V}$
is a continuous bijection of $\cV$ onto $\cT$.
Notice also that $\cV$ is compact and $\cT$ is Hausdorff.
Then we see that
$V\mapsto\overrightarrow{V}$
is a homeomorphism of $\cV$ onto $\cT$.
\end{proof}

%(see also \cite[Theorem~23.1]{Simon3})
%One can therefore define the homology class $[\overrightarrow{V}]\in H_3(M;\mathbb{Z})$
%by a theory of Federer~\cite[Chapter~4.4]{Federer}.
%Let $\cV$ be the set of all special Lagrangian varifolds $V$ of first variation $0$ 
%in $M.$
%Put $\cU_{f;a,b}=\{V\in\cV:a<V(f)<b\}$
%whenever $a<b$
%and $f$ is a continuous function with compact support in $G^\R_{\re\Om}(TM).$
%The family $\{\cU_{f;a,b}\}$ generates a topology on $\cV.$
%One easily proves that it is locally compact and Hausdorff.
%
%Let $V$ be a varifold of dimension $3$ in $M.$
%Put
%\begin{align}
%\|V\|(A)=V\bigl(\bigcup_{x\in A}G_\R(3;T_xM)\bigr)
%\end{align}
%for any subset $A$ of $M.$
%It is a Radon measure on $M.$

\section{Analysis of Blowing-up}
\l{bu}
In this section we shall analyse blowing-up near multiplicity-one
special Larangian cones with isolated singularity.
We can summarize our results as follows.

Let $(M,\om,J,\Om)$ be an almost Calabi--Yau manifold of complex dimension $m$, and $X$ a special Lagrangian $m$-fold in $(M,\om,J,\Om)$ with only one singular point $x$ modelled on a multiplicity-one special Lagrangian cone $C\su\C^m$ with isolated singularity, where $C$ need not be the stable $T^2$-cone as in \S\ref{intro} and so we are in a more general situation.
The problems in the present section are {\it local} near $x$ and so we may take $M$ to be an open ball about $x$.
We consider special Lagrangian varifolds $V_1,V_2,\cdots,V_n,\cdots$ tending to $X$.

We begin in \S\ref{energy section} by recalling the definition of an {\it energy} functional for $V_n$ introduced in the preceding paper \cite{Imagi}.
We prove that if the energy of $V_n$ is small then $V_n$ has singularity only at one point $y_n$ and
asymptotic at $y_n$ to
a multiplicity-one special Lagrangian cone $C^1$-close to $C$.

We also prove in \S\ref{bubble-off section} that if the energy of $V_n$ is large for all $n$ then
$V_n$ {\it blows up}; i.e. there exist points $y_n$ near $x_n$ and small numbers $s_n>0$ such that if we re-scale $V_n$ about $y_n$ by $s_n$ then it will tend to a special Lagrangian varifold $W$ with no boundary in $\C^m$ asymptotic at infinity to a multiplicity-one special Lagrangian cone $C^1$-close to $C$.

%Uhlenbeck's removable singularities theorem~\cite[Theorem~4.1]{Uhlenbeck};

\subsection{Energy Functional}
\l{energy section}
We define an energy functional as follows.
For each varifold $V$ of dimension $m$ in $\R^n$ we put
\[\mathcal{E}(V)=
\int_{\R^n\times G_m(\R^n)}
\frac{|S^{\perp}y|^2}{|y|^{m+2}}dV(y,S)\in[0,\iy]\]
where $S^\perp$ denotes the orthogonal complement to $S$ in $\R^n$,
and
$S^\perp y$ denotes the projection of $y$ onto $S^\perp$.

We shall recall a monotonocity formula for stationary varifolds.
For each $\rho>0$ let
$B_\rho$ denote the open ball of radius $\rho$ about $\bm0$ in $\C^m$, i.e. $B_\rho:=\{y\in\R^n:|y|<\rho\}$. 
For each $\rho>\si>0$ let $A_{\si,\rho}$ denote the open annulus of inner radius $\si$ and outer radius $\rho$ about $\bm0$ in $\C^m$, i.e.
$A_{\si,\rho}=B_\rho\-\,\ov{\!B_\si}$.
For each $Z\subset\R^n$ we put $\widetilde{Z}=Z\times G_m(\R^n)$.
If $V$ is a stationary varifold of dimension $m$ in $(B,g_0)$
and if $0<\si<\rho<1$
then we have
\begin{equation}\l{flat version}
\frac{\|V\|(B_\rho)}{\rho^m}
-\frac{\|V\|(B_\si)}{\si^m}
=\mathcal{E}(V\llcorner\widetilde{A}_{\si,\rho});
\end{equation}
for the proof we refer e.g. to Allard~\cite[Theorem~5.1(1)]{Allard}
or Simon~\cite[Equation~17.4]{Simon3}.
It is easy to extend \eqref{flat version} to Riemannian metrics in place of $g_0$:
\begin{proposition}\l{mono}
There exist constants $\ep_0\in(0,1)$ and $k>1$ depending only on $m,n$ and satisfying the following property:
let $g$ is a Riemannian metric on $B_1$
with $\ep:=|g-g_0|_{C^1(B_1)}<\ep_0.$
and $V$ a stationary varifold of dimension $m$ in $(B_1,g);$
then for every $\si,\rho\in\R$ with $0<\si<\rho<1$ we have
\e\l{sir}e^{k\ep\rho}\frac{\|V\|(B_\rho)}{\rho^m}
-e^{k\ep\si}\frac{\|V\|(B_\si)}{\si^m}
\geq\mathcal{E}(V\llcorner\widetilde{A}_{\si,\rho}).\e
\end{proposition}
\begin{rem}
The co-efficients $e^{k\ep\rho}$ and $e^{k\ep\si}$ on the left-hand side come from the following computation:
modifying the proof of Simon~\cite[Equation~17.4]{Simon3} we get indeed
\[k\ep\frac{\|V\|(B_\rho)}{\rho^m}+\frac{d}{d\rho}\frac{\|V\|(B_\rho)}{\rho^m}\ge\frac{d}{d\rho}\int_{y\in\widetilde{B}_\rho}\frac{|S^\perp y|^2}{|y|^{m+2}}dV\]
the left-hand side of which is not greater than
$\frac{d}{d\rho}\bigl(e^{k\ep\rho}\frac{\|V\|(B_\rho)}{\rho^m}\bigr)$
and so integration over the interval $(\si,\rho)$ implies \eq{sir}.
\end{rem}

We suppose now that $n=2m$ and $(\om,J,\Om)$
is an almost Calabi--Yau structure on $B_1$.
We denote by $\cV_0$ the space of all special Lagrangian
integral varifolds with no boundary in $(B_1,\om,J,\Om)$.

We define $r:\R^{2m}\to[0,\iy)$
by setting $r(y)=|y|$ for each $y\in\R^{2m}$.
We put $g_{\cyl}=r^{-2}g_0$,
which we shall call the cylindrical metric on $\R^{2m}\-\{0\}$.

We regard $(0,\iy)$ as a multiplicative group
acting upon $\R^{2m}$ as re-scaling.
By a smooth cone in $\R^{2m}$
we shall mean a closed submanifold of
$\R^{2m}\-\{0\}$
invariant under the re-scaling by $(0,\iy)$.
We denote by $\mathcal{C}$ the set of all special
Lagrangian smooth cones in $(\R^{2m},\om_0,J_0,\Om_0)$.

We suppose $C\in\mathcal{C}$.
We denote by $NC$ the normal bundle to $C$ in $\R^{2m}$
with respect to $g_0$.
We get the same bundle
even if we use $g_{\cyl}$ in place of $g_0$.

For each $u\in C^k(C;NC)$ and $\rho<\si$
we define $|u|^{k,\cyl}_{[\rho,\si]}$ as follows.
We put $t=-\log r$ and
$\d_t=\frac{\d}{\d t}$.
We put $\Si=C\cap S^{2m-1}$.
By the definition of smooth cones
$\Si$ is a compact submanifold of $S^{2m-1}$.
We denote by $\nabla_\Si$ the Levi-Civita connexion over $\Si$
induced from $g_0$.
We put
\[|u|^{k,\cyl}_{[\rho,\si]}=
\sup_{C\cap A_{\si,\rho}}
\sum_{i,j\geq0,i+j\leq k}|\d_t^i\nabla_\Si^ju|.\]

If $|u|^{1,\cyl}_{[\rho,\si]}$ is sufficient small,
then we can define
the exponential map $\exp u:C\cap A_{\si,\rho}\to A_{\si,\rho}$
with respect to the metric $g_{\cyl}$,
and the image of $\exp u$ will be a submanifold of $A_{\si,\rho}$,
which we shall denote by $\gr_{\cyl}u$.
We put $|{\gr_{\cyl}u}|=\mathcal{H}^m\llcorner\gr_{\cyl}u$
where $\mathcal{H}^m$ denotes the Hausdorff $m$-dimensional measure
with respect to the almost Calabi--Yau metric $g$.

We suppose that $(\om,J,\Om)$ is an almost Calabi--Yau structure on $B_1$,
and $g$ is the almost Calabi--Yau metric on $B_1$.
From the proof of the author~\cite[Theorem~2.2]{Imagi}
we get:
\begin{thm}\l{AAS}
There exists $\ep>0$ depending only on $m$ and $C$
such that if we have
\[|\Om-\Om_0|_{C^0(B_1)}+
|g-g_0|_{C^1(B_1)}<\ep,\;
V\in\cV_0,\;0<\rho<1,\;
\mathcal{E}(V\llcorner\widetilde{A}_{\rho,1})<\ep\]
\[\text{and }
v\in C^\iy(C\cap A_{1/2,1};NC),\;|v|^{2,\cyl}_{[1/2,1]}<\ep,\;
|{\gr_{\cyl}v}|=\|V\|\llcorner A_{1/2,1}\]
then we can extend $v$ to $C\cap A_{\rho,1}$
so that $\|V\|\llcorner A_{\rho,1}=|{\gr_{\cyl}v}|$
and $|v|^{2,\cyl}_{[\rho,1]}<k\ep^\al$
for some $k>0$ and $\al\in(0,1)$ depending only on $m$ and $C$.
\end{thm}
\begin{rem*}
For the proof one has to use a result of \L{}ojasiewicz~\cite{Lojasiewicz},
following Simon~\cite{Simon}.
\end{rem*}
We shall give a corollary to the theorem above.
We give $\mathcal{C}$ the $C^\iy$-topology.
%by Allard's regularity theorem, actually,
%it is the same as the one induced from $\cV_0$.
Let $\mathcal{C}'$ be a neighbourhood of $C$ in $\mathcal{C}$.
Then we have:
\begin{cor}\l{energy small}
There exists $\ep>0$ depending only on $m$, $C$ and $\mathcal{C}'$
such that if we have
\[|\Om-\Om_0|_{C^0(B_1)}+
|g-g_0|_{C^1(B_1)}<\ep,\;
V\in\cV_0,\;\mathcal{E}(V)<\ep\]
\[\text{and }
v\in C^\iy(C\cap A_{1/2,1};NC),\;|v|^{2,\cyl}_{[1/2,1]}<\ep,\;
|{\gr_{\cyl}v}|=\|V\|\llcorner A_{1/2,1}\]
then $V$ is singular only at $\bm0$
and asymptotic at $\bm0$ to some $C'\in\mathcal{C}'$ with multiplicity $1$.
\end{cor}
\begin{rem}\l{en sm}
By a result of Simon~\cite[Theorem~5]{Simon}
$C'$ will be a unique tangent cone to $V$ at $\bm0$.
\end{rem}
\begin{proof}
Let $C'$ be a tangent cone to $V$ at $\bm0$.
For each $\de>0$ let $\de^{-1}V$ denote the re-scaling of $V$ by $\de^{-1}$.
Then by the definition of tangent cones we can take $\de_1>\de_2>\de_3>\cdots$
tending to $0$ with $\de_n^{-1}V$ tending to $C'$
as $n\to\iy$.
By definition $\cC'$ is a $C^1$-neighbourhood of $C$ in $\cC$ and so we can take $\eta>0$ such that if
$v\in C^\iy(\Si;NC)$ and $|v|_{C^1(\Si)}\leq\eta$
then we have $\gr_{\cyl}v\in\mathcal{C}'$.

Let $k,\al$ be as in Theorem \ref{AAS} and make $\ep$ so small that $k\ep^\al<\eta$.
Applying Theorem \ref{AAS} with $\rho=\de/2$ we find $v_n\in C^\iy(C\cap A_{\de_n/2,1};NC)$
such that
\e\l{vn}\|V\|\llcorner A_{\de_n/2,1}
=|{\gr_{\cyl}v_n}|\text{ with }|v_n|^{2,\cyl}_{[\de_n/2,1]}\leq\eta.\e
Define $v_n'\in C^\iy(C\cap A_{1/2,1};NC)$ by $v_n'(y)=v_n(\de_n y)$ for $y\in C\cap A_{1/2,1}$. Then by \eq{vn} we have
\e\|\de_n^{-1}V\|\llcorner A_{1/2,1}
=|{\gr_{\cyl}v_n'}|\text{ with }|v_n'|^{2,\cyl}_{[1/2,1]}\leq\eta.\e
The definition of $|\bu|^{2,\cyl}_{[1/2,1]}$ also implies that $v_n'$ is $C^2$-bounded and so we can find a subsequence of $v_n'$ converging in the $C^1$-topology as $n\to\iy$ to some $w\in C^1(C\cap A_{1/2,1};NC)$ with $|w|^{1,\cyl}_{[1/2,1]}\le\eta$.
On the other hand since $\de_n^{-1}V$ tends to $C'$ as $n\to\iy$ it follows that
\e\l{cw}C'\cap A_{1/2,1}=\gr w.\e
This implies that $C'$ is a multiplicity-one cone with isolated singularity, and a result of Simon~\cite[Theorem~5]{Simon} implies in particular that $V$ is singular only at $\bm0$.

The equation \eq{cw} also implies that $w$ is $C^\iy$-differentiable and invariant under the re-scaling of $C$. In particular since $|w|^{1,\cyl}_{[1/2,1]}\le\eta$ it follows that $C'\in\cC'$, completing the proof of Corollary \ref{energy small}.
\end{proof}

\subsection{Bubbling-off}
\l{bubble-off section}
We suppose that $(\om,J,\Om)$ is an almost Calabi--Yau structure on $B_1$
with $J|_0=J_0$ and $\Om|_0=\Om_0$.
We denote by $g$ the almost Calabi--Yau metric on $(B_1,\om,J,\Om)$.
We write $g=\sum_{i,j=1}^{2m}g_{ij}dy^idy^j$
and suppose
\begin{equation*}
g_{ij}(0)=\de_{ij}\text{ and }
\frac{\d g_{ij}}{\d y^k}(0)=0
\text{ for each }i,j,k=1,\cdots,2m.
\end{equation*}

If $W$ is a varifold in $\R^n$
and if $\de>0$ then we can define $\de^{-1}W$
by re-scaling $W$ by $\de^{-1}$ (as in the proof of Corollary \ref{energy small} above).
If $b\in\R^n$ then
we can also define $W-b$
by translating $W$ by $-b$.

We take $C\in\mathcal{C}$
and a neighbourhood $\mathcal{C}'$ of $C$ in $\mathcal{C}$.
We denote by $\mathcal{X}'$
the space of all elements of $\cV_0$ with singularity only at one point $y$
and asymptotic at $y$ to some element of $\mathcal{C}'$
with multiplicity $1$.
Let $X\in\cX'$ and let $X$ be singular at $\bm0$ and asymptotic at $\bm0$ to $C$ with multiplicity $1$.
We have then:
\begin{thm}\l{X top}
Let $(X_n)_{n=1}^\iy$ be a sequence in $\mathcal{X}'$
converging to $X$.
For each $n=1,2,3,\cdots$
let $x_n$ be the singular point of $X_n$,
and let $C_n$ be the multiplicity $1$ smooth tangent cone to $X_n$ at $x_n$.
Then $x_n$ tends to $\bm0$ and $C_n$ tends to $C$ as $n\to\iy$.
\end{thm}
\begin{proof}
By Allard's regularity theorem $x_n$ tends to $\bm0$ as $n\to\iy$, and so we have only to prove that for each neighbourhood $\cC'$ of $C$ in $\cC$ there exists an integer $N>0$ such that for $n>N$ we have $C_n\in\cC$.
Let $\ep>0$ be so small that we may apply Proposition \ref{mono} and Corollary \ref{energy small} where Corollary \ref{energy small} involves the choice of $\cC'$ and so $\ep$ depends on $\cC'$.
By Proposition~\ref{mono} we have
\[\mathcal{E}\bigl((X_n-x_n)\llcorner\widetilde{A}_{\si,\rho}\bigr)\leq e^{k\ep\rho}\frac{\|X_n-x_n\|(B_\rho)}{\rho^m}
-e^{k\ep\si}\frac{\|X_n-x_n\|(B_\si)}{\si^m}\]
where $0<\si<\rho<1$.
Letting $\si\to0$ we get
\e\l{EX}\mathcal{E}(X_n-x_n)\leq e^{k\ep\rho}\frac{\|X_n-x_n\|(B_\rho)}{\rho^m}
-\area(C_n\cap B_1).\e
Making $\cC'$ smaller if necessary we may suppose that $|\area(C'\cap B_1)-\area(C\cap B_1)|<\ep/2$ for all $C'\in\cC'$ so that by \eq{EX} we have
\begin{equation}\l{EE}
\mathcal{E}(X_n-x_n)\leq e^{k\ep\rho}\frac{\|X_n-x_n\|(B_\rho)}{\rho^m}
-\area(C\cap B_1)+\frac{\ep}{2}.
\end{equation}
Making $\rho$ smaller if necessary we may suppose
\[e^{k\ep\rho}\frac{\|X\|(B_\rho)}{\rho^m}-\area(C\cap B_1)<\frac{\ep}{2}.\]
For $n$ sufficiently large, therefore, we have
\[e^{k\ep\rho}\frac{\|X_n-x_n\|(B_\rho)}{\rho^m}-\area(C\cap B_1)<\frac{\ep}{2}.\]
This combined with \eqref{EE} implies
\[\mathcal{E}(X_n-x_n)<\ep.\]
On the other hand Allard's regularity theorem implies that $\|X_n-x_n\|\llcorner A_{1/2,1}=\gr v_n$ for some $v_n\in C^\iy(C\cap A_{1/2,1};NC)$ with $|v_n|_{[1/2,1]}^{2,\cyl}<\ep$.
Hence applying Corollary~\ref{energy small} and Remark \ref{en sm} to $X_n-x_n$ we find that $C_n\in\cC'$, completing the proof of Theorem \ref{X top}.
\end{proof}

We have also:
\begin{thm}\l{bubble-off}
Let $(V_n)_{n=1}^\iy$ be a sequence in
$\cV_0\-\mathcal{X}'$
converging to $X$.
Then there exists a sequence $(\de_n)_{n=1}^\iy$
of positive real numbers converging to $0$,
a sequence $(y_n)_{n=1}^\iy$ in $B_1$
converging to $0$,
and a subsequence of $\bigl(\de_n^{-1}(V_n-y_n)\bigr)_{n=1}^\iy$
converging to some varifold $W$
in $\R^{2m}$
asymptotic at infinity to some element of $\mathcal{C}'$ with multiplicity $1$
and satisfying $\mathcal{E}(W-b)>0$ for every $b\in\R^{2m}$.
\end{thm}
\begin{rem*}
$W$ will automatically
be a special Lagrangian integral varifold
with $\d{\overrightarrow{W}}=0$
in $(\R^{2m},\om_0,J_0,\Om_0)$.
\end{rem*}
\begin{proof}
We take positive real numbers $\si,\rho$ and $\ep$,
which we shall make smaller if necessary.
Allard's regularity theorem implies that for each $y\in B_{\si\rho}$ and $n$ large enough there exists $v_{n,y}\in C^\iy(C\cap A_{1/2,1};NC)$ such that
\e\l{vy}\|V_n-y\|\llcorner A_{1/2,1}=|\gr v_{n,y}|,\,|v_{n,y}|_{[1/2,1]}^{2,\cyl}<\ep\e
where $\ep$ is as in Corollary \ref{energy small}.
Hence by Corollary \ref{energy small} we find that $\mathcal{E}(V_n-y)\geq\ep$;
otherwise $V_n$ will be singular only at $y$
and asymptotic at $y$ to some element of $\mathcal{C}'$
with multiplicity $1$, which contradicts $V_n\in\cV_0\-\mathcal{X}'$.
We put
\[\de_n(y)
=\inf\left\{\de\in(0,\rho):\mathcal{E}\bigl((V_n-y)
\llcorner\widetilde{A}_{\de,\rho}\bigr)=\frac{\ep}{2}\right\}.\]
Since $\mathcal{E}(V_n-y)\geq\ep$ we see $\de_n(y)>0$.
It is also easy to see that
$y\mapsto\de_n(y)$ is lower semi-continuous.
Hence we can find
$y_n\in B_{\si\rho}$ with
$\de_n(y_n)=\inf_{y\in B_{\si\rho}}\de_n(y)$.
We put
$\de_n=\de_n(y_n)>0$.
We have then:
\begin{lem*}
$\de_n$ tends to $0$ as $n\to\iy$.
\end{lem*}
\begin{proof}
Making $\rho$ smaller if necessary
we may suppose $\mathcal{E}(X\llcorner\widetilde{B}_\rho)<\ep/2$.
We take $\de>0$.
We have then $\mathcal{E}(X\llcorner\widetilde{A}_{\de,\rho})<\ep/2$.
Since $V_n$ tends to $X$ as $n\to\iy$
we see that
for $n$ sufficiently large
we have $\mathcal{E}(V_n\llcorner\widetilde{A}_{\de,\rho})<\ep/2$.
This implies $\de_n(0)<\de$
and so $\de_n(0)$ tends to $0$ as $n\to\iy$.
Since $\de_n\leq\de_n(0)$
we see that $\de_n$ also tends to $0$
as $n\to\iy$.
\end{proof}
We have also:
\begin{lem*}
$\bigl(\de_n^{-1}(V_n-y_n)\bigr)_{n=1}^{\iy}$
has a subsequence converging to
some varifold $W$ in
$\R^{2m}$.
\end{lem*}
\begin{proof}
Let $R>0$. Then we have only to prove
\begin{equation}\l{bdd}
\sup_{n=1,2,3,\cdots}\|\de_n^{-1}(V_n-y_n)\|(B_R)<\iy.
\end{equation}
Notice that
$\|\de_n^{-1}(V_n-y_n)\|(B_R)=\de_n^{-m}\|V_n-y_n\|(B_{\de_nR})$.
Notice also that by Proposition~\ref{mono} we can find $\ep'>0$ such that
\[e^{k\ep'\de_nR}\frac{\|V_n-y_n\|(B_{\de_nR})}{(\de_nR)^m}
\leq e^{k\ep'}\|V_n-y_n\|(B_1)\]
which tends to $e^{k\ep'}\|X\|(B_1)$ as $n\to\iy$.
Then we get \eqref{bdd}.
\end{proof}

We take $W$ as above.
From the definition of $\de_n$
it is easy to see
$\mathcal{E}(W\llcorner\widetilde{A}_{1,\iy})\leq\ep/2$.
Using \eq{vy} with $y=y_n$ we can also apply Theorem \ref{AAS} to find $v_n\in C^\iy(C\cap A_{\de_n,1};NC)$ with
\e\|V_n-y_n\|\llcorner A_{\de_n,1}=\gr v_n,\,|v_n|_{[\de_n,1]}^{2,\cyl}<k\ep^\al.\e
This implies that there exists $w\in C^\iy(C\cap A_{1/2,1};NC)$ such that
\e\|W\|\llcorner A_{1,2}=\gr w,\,|w|_{[1,2]}^{2,\cyl}\le k\ep^\al\e
and so the situation is similar to that of Corollary~\ref{energy small} in the sense that $W$ is away from infinity close to the multiplicity-one cone with isolated singularity and that $W$ satisfies the energy estimate $\mathcal{E}(W\llcorner\widetilde{A}_{1,\iy})\leq\ep/2$.
In a way similar to the proof of Corollary \ref{energy small} we can prove indeed that
$W$ is asymptotic at infinity to a multiplicity-one cone $C'$ with isolated singularity.
Making $\ep$ smaller if necessary we may suppose that $C'$ is so $C^1$-close to $C$ that
\e\l{CC'}|\area(C'\cap B_1)-\area(C\cap B_1)|<\frac{\ep}{8}.\e

It remains to prove $\mathcal{E}(W-b)>0$
for every $b\in\R^{2m}$.
We put $V_n'=V_n-y_n-\de_nb$.
By Proposition~\ref{mono} we have
\begin{equation}\l{V_n'}
\exp(k\ep'\rho)
\frac{\|V_n'\|(B_\rho)}{\rho^m}
-\exp(k\ep'\de_n)\frac{\|V_n'\|(B_{\de_n})}{\de_n^m}
\geq\mathcal{E}(V_n'\llcorner\widetilde{A}_{\de_n,\rho})\geq\frac{\ep}{2}.
\end{equation}
We note that $V_n'$ tends to $X$ and $\de_n^{-1}V_n'$ tends to $W-b$
as $n\to\iy$.
By \eqref{V_n'} therefore we have
\begin{equation}\l{V_n''}
e^{k\ep'\rho}\frac{\|X\|(B_\rho)}{\rho^m}
-\|W-b\|(B_1)\geq\frac{\ep}{2}.
\end{equation}
Making $\rho>0$ smaller if necessary we may suppose
\[e^{k\ep'\rho}\frac{\|X\|(B_\rho)}{\rho^m}
\leq\area(C\cap B_1)+\frac{\ep}{4}.\]
This combined with \eqref{V_n''} implies
\[
\area(C\cap B_1)-\|W-b\|(B_1)\ge\frac{\ep}{4}.
\]
Hence by \eq{CC'} we get
\begin{equation}\l{V_n3}
\area(C'\cap B_1)-\|W-b\|(B_1)>\frac{\ep}{8}.
\end{equation}
On the other hand by \eqref{flat version} we have
\[\frac{\|W-b\|(B_R)}{R^m}-\|W-b\|(B_1)
=\mathcal{E}\bigl((W-b)\llcorner\widetilde{A}_{1,R}\bigr)\]
for each $R>1$. Letting $R\to\iy$ we get
\[
\area(C'\cap B_1)-\|W-b\|(B_1)
=\mathcal{E}\bigl((W-b)\llcorner\widetilde{A}_{1,\iy}\bigr).
\]
This combined with \eqref{V_n3} implies
$\mathcal{E}\bigl((W-b)\llcorner\widetilde{A}_{1,\iy}\bigr)>\ep/8$
and so $\mathcal{E}(W-b)>0$, completing the proof.
\end{proof}

\section{Classification of Local Models}
\l{un}
Let $C$ be as in \eq{C}, and let $L_1,L_2,L_3$ be as in \eq{L1}--\eq{L3}.
We begin by recalling the statement of our classification result:
\begin{thm}[Re-statement of Theorem \ref{uniq}]\l{uniqueness}
Let $W$ be a special Lagrangian integral varifold with no boundary in $(\R^6,\om_0,J_0,\Om_0)$
asymptotic at infinity to $C$ with multiplicity $1$.
Then we have $\|W\|=|sL+b|$
for some $s>0$, $L\in\{C,L_1,L_2,L_3\}$ and $b\in\R^6$.
\end{thm}
\noindent
Here $|sL+b|$ denotes the Radon measure on $\R^6$
associated to $sL+b$
with multiplicity $1$ with respect to the Euclidean metric $g_0$.

In what follows we give a sketch of the proof of Theorem~\ref{uniqueness}.

First of all we analyse the asymptotic behaviour of
$W$ at infinity.
From the stability of $C$
we find $b\in\R^6$ such that
$W-b$ will decay rapidly to $C$ at infinity.
For the sake of simplicity, therefore,
we may suppose that
$W$ decays rapidly to $C$ at infinity.

We define a $T^2$-action on $\R^6=\C^3$ by setting
\[(e^{i\theta},e^{i\phi})\cdot(z_1,z_2,z_3)
=(e^{i\theta}z_1,e^{i\phi}z_2,e^{-i\theta-i\phi}z_3).\]
We note that $\om_0$, $J_0$, $\Om_0$ and $C$ are invariant under the
$T^2$-action.
We take a moment map $\mu:\C^3\to\R^2$.
We prove that $\mu$ is constant on $\spt\|W\|$.
The idea of the proof is as follows.

We suppose for a moment that $W$ is a submanifold of $\C^3$.
By the $T^2$-action, then, we can deform $W$
as a special Lagrangian submanifold of $\C^3$.
The infinitesimal deformation may be identified with $d\mu|_W$,
and so $d\mu|_W$ will be a harmonic $1$-form on $W$, by the theory of McLean.
Therefore $\mu|_W$ will be a harmonic function on $W$.
Since $W$ decays rapidly to $C$ at infinity we shall see that
$\mu|_W$ decays rapidly at infinity to some constant.
By the maximum principle, therefore,
$\mu|_W$ will be a constant function.

Actually $W$ need not be a submanifold of $\R^6$.
Using some basic results on varifolds, however, we can modify the argument above.
Thus we see that
$\mu$ is constant on $\spt\|W\|$.

Harvey and Lawson
\cite[Chapter~III.3.A, Theorem~3.1]{HL}
construct a special Lagrangian fibration
$F:\C^3\to\R^3$
invariant under the $T^2$-action.
From the proof of Harvey and Lawson
%or a more general result of Goldstein~\cite[Theorem~A]{Goldstein}
we see that every $T^2$-invariant special Lagrangian submanifold of $\C^3$
is contained in a fibre of $F$.

We suppose again that
$W$ is a submanifold of $\C^3$.
Since $\mu$ is constant on $W$ we see that
$W$ is then invariant under the $T^2$-action,
and so contained in a fibre of $F$. 
Actually $W$ need not be a submanifold of $\C^3$,
but we can again modify the argument so that
$\spt\|W\|$ will be contained in a fibre of $F$.

For each fibre of $F$ we have an explicit description of
the topology and asymptotic behaviour at infinity,
and so we shall be able to
complete the proof by an elementary argument.

We begin now with a review of some basic properties of
Laplacians over cones.

We suppose that $\Si$ is a compact manifold of dimension $m-1$.
We put $C=(0,\iy)\times\Si$.
We denote by $r$ the projection of $C$ onto $(0,\iy)$.
We suppose that $g_\Si$ is a Riemannian metric on $\Si$.
We put $g_C=dr^2+r^2g_\Si$.
With respect to $g_C$
we can define the Laplacian $\Delta_C:C^\iy(C;\R)
\to C^\iy(C;\R)$.

We have also the Laplacian $\Delta_\Si:C^\iy(\Si;\R)\to
C^\iy(\Si;\R)$ with respect to $g_\Si$.
We put
$t=\log r:C\to\R$ and $\d_t=\frac{\d}{\d t}$.
It is easy then to see
$-e^{2t}\Delta_C=\d_t^2+(m-2)\d_t-\Delta_\Si$.
We denote by
$0=\ga_0\leq\ga_1\leq\ga_2\leq\cdots$ the eigenvalues of $\Delta_\Si$.
For each integer $i\geq0$
we consider the equation $x^2+(m-2)x-\ga_i=0$
in $x$.
We denote by $\al_i$ and $\be_i$
the two solutions with $\al_i\geq\be_i$.
We suppose $m\geq2$.
Since $\ga_0=0$ we get $\be_0=2-m\leq0=\al_0$.
Since $\ga_0\leq\ga_1\leq\ga_2\cdots$
we get $\al_0\leq\al_1\leq\al_2\leq\cdots$
and $\be_0\geq\be_1\geq\be_2\geq\cdots$.
Since $\ga_i$ tends to $\iy$
as $i\to\iy$ we see that $\al_i$ tends to $\iy$
and $\be_i$ tends to $-\iy$ as $i\to\iy$.
We also put
$\Lambda=\{\al_0,\be_0,\al_1,\be_1,\al_2,\be_2,\cdots\}$.
For each $\lambda\in\R$ we put
\[E_\lambda=\{f\in C^{\iy}(\Si;\R):
\Delta_{\Si}f=\lambda(\lambda+m-2)f\}.\]
By the definition of the eigenvalues we have
$E_\lambda\neq\{0\}$ if and only if $\lambda\in\Lambda$.

We can also take a complete orthonormal basis
$\{v_0,v_1,v_2,\cdots\}$ for $L^2(\Si;\R)$ such that
$\Delta_{\Si}v_i=\ga_iv_i$ for each integer $i\geq0$.
We have then:
\begin{proposition}\l{homogeneous}
Let $I$ be an open interval in $\R$,
let $u\in C^\iy(I\times \Si;\R)$
and suppose $\Delta_Cu=0$.
Then there exist $a_0,b_0,a_1,b_1,a_2,b_2\cdots\in\R$
such that
$u=\sum_{i=0}^\iy(a_ir^{\al_i}+b_ir^{\be_i})v_i$
where the series converges in the local $C^\iy$-topology.
\end{proposition}
We give a proof for the sake of clarity:
\begin{proof}
We put $\log I=\{t\in\R:e^t\in I\}$.
For each $t\in\log I$ and $i\in\{0,1,2,\cdots\}$ we define
$u_i(t)$ as the inner product of $u|_{\{e^t\}\times\Si}$
and $v_i$ in $L^2(\Si;\R)$.
Then $t\mapsto u_i(t)$ is a smooth function on $\log I$.
Since $\Delta_Cu=0$ we get
$(\d_t^2+(m-2)\d_t-\lambda_i)u_i=0$.
Hence we find $a_i,b_i\in\R$ such that
$u_i=a_ie^{\al_it}+b_ie^{\be_it}$.
It suffices therefore to prove that $\sum_{i=0}^\iy u_iv_i$
converges to $u$ in the local $C^\iy$-topology.
We denote by $d\mu_\Si$ the Riemannian measure on $\Si$ with respect to $g_\Si$,
and by $\|\bullet\|$ the norm of $L^2(\Si;\R)$.
We have then
\[\int_{\log I\times\Si}|u|^2dtd\mu_\Si=\int_{\log I}\sum_{i=0}^\iy\|u_i(t)\|^2dt
=\sum_{i=0}^\iy\int_{\log I}\|u_i(t)\|^2dt<\iy.\]
Putting $w_n=\sum_{i=0}^nu_iv_i$ we get
$\int_{\log I\times\Si}|u-w_n|^2dtd\mu_\Si=\sum_{i=n+1}^\iy\int_{\log I}
\|u_i(t)\|^2dt$
which tends to $0$ as $n\to\iy$.
Thus $w_n$ tends to $u$ in the $L^2$-topology.
Applying elliptic regularity to $u-w_n$
we see that $w_n$ tends to $u$ in
the local $C^\iy$-toplogy.
\end{proof}

Let $u\in C^\iy(C;\R)$
and $\al\in\R$.
%we shall write $u=O(r^\al)$
%as $r\to0$ if we have
%\[\lim_{\ep\to+0}\sup_{C\cap B_\ep}|r^{-\al+k}\nabla^ku|<\iy
%\text{ for every }k=0,1,2,3,\cdots\]
%where $\nabla$ denotes the Levi-Civita connexion
%over the Riemannian manifold $(C,i_C^*g_0)$,
%and $|\bullet|$ denotes the norm with respect to $i_C^*g_0$.
Then we shall write $u=O(r^\al)$ as $r\to\iy$ if
there exists $R>0$ such that
\[\sup_{(R,\iy)\times\Si}|r^{-\al+k}\nabla^ku|<\iy
\text{ for every }k=0,1,2,3,\cdots\]
where $\nabla$ denotes the Levi-Civita connexion with respect to $g_C$,
and $|\bullet|$ denotes the norm with respect to $g_C$.

By a result of Simon~\cite[Part~I, Lemma~5.9]{Simon2}
we have:
\begin{proposition*}
Suppose
$q\in(2-m,\iy)\-\Lambda$ and
$f\in C^\iy\bigl((R,\iy)\times\Si;\R\bigr)$
with $f=O(r^{q-2})$.
Then there exists $u\in C^\iy\bigl((R,\iy)\times\Si;\R\bigr)$
with $u=O(r^q)$ such that $\Delta_Cu=f$.
\end{proposition*}
We give a corollary to this:
\begin{cor}
\l{linearized equation}
Suppose $R>0$, $u\in C^{\iy}\bigl((R,\iy)\times\Si;\R\bigr)$,
$p,q\in\R$, $2-m<q<p$, $q\notin\Lambda$,
$u=O(r^p)$ and
$\Delta_Cu=O(r^{q-2})$.
Then
there exists
$(f_\lambda)_{\lambda\in\Lambda}\in\bigoplus_{\lambda\in\Lambda}
E_\lambda$
such that $u=
\sum_{\lambda\in\Lambda\cap(q,p]}f_\lambda r^\lambda+O(r^q)$
where $\Lambda\cap(q,p]$ is possibly empty
and in that case we set
$\bigoplus_{\lambda\in\emptyset}E_\lambda=\{0\}$.
\end{cor}
\begin{proof}
Applying the proposition above to $\Delta_Cu$ in place of $f$
we find $u'\in C^\iy\bigl((R,\iy)\times\Si;\R\bigr)$
with $u'=O(r^q)$ such that $\Delta_Cu'=\Delta_Cu$.
By Proposition~\ref{homogeneous}
we can find some $a_0,b_0,a_1,b_1,a_2,b_2,\cdots\in\R$ such that
$u-u'=\sum_{i=0}^\iy(a_ir^{\al_i}+b_ir^{\be_i})v_i$.

Since $u'=O(r^q)$ we get
$u=\sum_{i=0}^\iy(a_ir^{\al_i}+b_ie^{\be_i})v_i+O(r^q)$.
Since $u=O(r^p)$ and $p>q$
we get $a_i=0$ if $\al_i>p$.
Since $\al_0\leq\al_1\leq\al_2\leq\cdots$
tend to $\iy$
we can find a unique integer $i(p)$ such that $\al_{i(p)+1}>p$
and $\al_{i(p)}\leq p$.
Since $a_i=0$ for every $i>i(p)$
we see that
$\sum_{i=0}^\iy b_ir^{\be_i}v_i$
converges in the local $C^\iy$-topology.
We put $w=\sum_{i=0}^\iy b_ir^{\be_i}v_i$.

Since $0\geq2-m=\be_0\geq\be_1\geq\be_2\geq\cdots$
it follows that for every $\rho>R$ we have
\[\int_\rho^{2\rho}\int_\Si
|w|^2\frac{dr}{r}d\mu_\Si\leq\left(\frac{\rho}{R}\right)^{2-m}
\int_R^{2R}\int_\Si|w|^2\frac{dr}{r}d\mu_\Si.\]
Applying elliptic regularity to $w$
we get $w=O(r^{2-m})$.
Since
$u=\sum_{i=0}^{i(p)}a_ir^{\al_i}v_i
+w+O(r^q)$ and $q>2-m$
we get $u=\sum_{i=0}^{i(p)}a_ir^{\al_i}v_i
+O(r^q)$, completing the proof.
\end{proof}
We suppose now
that $C$ is a smooth special Lagrangian cone in
$(\R^{2m}\-\{0\},\om_0,J_0,\Om_0)$
and $\Si=C\cap S^{2m-1}$.
For each $\rho>0$ we put $B_\rho=\{y\in\R^n:|y|<\rho\}$.
We denote by $i_C:C\to\R^{2m}$ the inclusion map of $C$ into $\R^{2m}$.

Joyce~\cite[Definition~7.1]{J1} defines
special Lagrangian submanifolds of $(\R^{2m},\om_0,J_0,\Om_0)$
asymptotic to $C$ with multiplicity $1$ at infinity with some rate $<2$.
We extend it to varifolds as follows.
We denote by $\mathcal{W}$ the set of all
special Lagrangain integral varifolds $W$
with $\d\overrightarrow{W}=0$ in
$(\R^{2m},\om_0,J_0,\Om_0)$
asymptotic at infinity to $C$ with multiplicity $1$.
For each $\lambda<2$ we denote by $\mathcal{W}_\lambda$
the set of all $W\in\mathcal{W}$
such that we can find a compact subset $K$ of $W$,
an $R>0$ and a diffeomorphism $f:C\-\,\ov{\!B_R}\to
\spt\|W\|\- K$ such that $f-i_C=O(r^{\lambda-1})$.

We suppose that $C$ is Jacobi-integrable in the sense of
Joyce~\cite[Definition~6.7]{J1}.
In a way similar to Simon~\cite[Part~II, \S\S5 and 6]{Simon2},
then, we can prove that if $W\in\mathcal{W}$ then
there exists $\lambda<2$ such that $W\in\mathcal{W}_\lambda$.

Joyce~\cite[Theorem~4.3]{J1} proves
a version of Weinstein's theorem~\cite[Corollary~6.2]{W},
which we shall recall next.
We denote by $T^*C$ the cotangent bundle over $C$,
by $0_C$ the zero-section of $T^*C$, and by $\om_C$
the canonical symplectic form on $T^*C$.
We regard $(0,\iy)$ as a multiplicative group
acting upon $C$ and $\R^{2m}$ as re-scaling.
We can lift the $(0,\iy)$-action
uniquely to $T^*C$
so that for each $t\in(0,\iy)$ we shall have
$t^*\om_C=t^2\om_C$ (on the left-hand side
we regard $t$ as a map of $C$ into itself).
We have then:
\begin{lem}\l{UC}
There exist a neighbourhood $UC$ of $\im0_C$ in $T^*C$
invariant under $(0,\iy)$, 
and a diffeomorphism $\Phi_C$ of $UC$ into $\R^{2m}$
equivariant under $(0,\iy)$
with $\Phi_C\circ0_C=i_C$ and $\Phi_C^*\om_0=\om_C$.
\end{lem}

Let $W\in\mathcal{W}_\lambda$.
Then we can take a compact subset $K$ of $W$
and a closed $1$-form $w$ on $C\-\,\ov{\!B_R}$
such that $\spt\|W\|\- K_W=\Phi_C(\gr w)$.
We denote by $\pi_\Si$ the projection of $C$ onto $\Si$, which induces a linear isomorphism $\pi_\Si^*:H^1(\Si;\R)\to H^1(C\-\,\ov{\!B_R};\R)$,
so that we may write $w=\pi_\Si^*\eta_W+dh_W$
for some $1$-form $\eta_W$ on $\Si$
and some $h_W:C\-\,\ov{\!B_R}\to\R$.
By results of Joyce~\cite[Equations~(7.7) and (7.8)]{J1} we have:
\begin{lem}\l{delta h}
If $\al<2$ and $h_W=O(r^\al)$ then
we have $\Delta_Ch_W=O(r^{2(\al-2)})$.
\end{lem}
We can extend a result of Joyce~\cite[Theorem~7.11]{J1}
as follows:
\begin{lem}\l{improvement}
Let $W\in\mathcal{W}_\lambda$,
$\lambda'<\lambda<2$ and $[\lambda',\lambda)\cap\Lambda=\emptyset$.
Then we have $W\in\mathcal{W}_{\lambda'}$.
\end{lem}
\begin{proof}
For each integer $n\geq0$ we put
$\lambda(n)=2^n(\lambda-2)+2$.
We can take a unique integer $\nu$ such that
$\lambda(\nu+1)<\lambda'\leq\lambda(\nu)$.
By an induction on $n=0,1,\cdots,\nu$,
we shall prove
$h_W=O(r^{\lambda(n)})$ for every $n=0,1,\cdots,\nu$.

By the property of $h_W$ we have $h_W=O(r^\lambda)=O(r^{\lambda(0)})$.
If $\nu=0$ we can then complete the induction automatically.
We suppose therefore $\nu>0$.
Suppose also that we have $h_W=O(r^{\lambda(n)})$
for some $n=0,1,\cdots,\nu-1$.
By Lemma~\ref{delta h} we have $\Delta h_W=O(r^{2(\lambda(n)-2)})=O(r^{\lambda(n+1)-2})$.
Since $n<\nu$ we get $\lambda(n+1)\geq\lambda'$ and so
$[\lambda(n+1),\lambda)\cap\Lambda=\emptyset$.
Applying Corollary~\ref{linearized equation} to
$h_W,\lambda,\lambda(n+1)$ in place of $u,p,q$ respectively, we get
$h_W=O(r^{\lambda(n+1)})$, completing the induction.

We have thus proved
$h_W=O(r^{\lambda(n)})$ for every $n=0,1,\cdots,\nu$.
Putting $n=\nu$ we get $h_W=O(r^{\lambda(\nu)})$.
By Lemma~\ref{delta h} we have $\Delta_Ch_W=O(r^{\lambda(\nu+1)-2})$.
By the definition of $\nu$ we have $\lambda(\nu+1)\leq\lambda'$
and so $\Delta_Ch_W=O(r^{\lambda'-2})$.
Applying Corollary~\ref{linearized equation} to
$h_W,\lambda,\lambda'$ in place of $u,p,q$ respectively, we get
$h_W=O(r^{\lambda'})$, completing the proof.
\end{proof}
Joyce~\cite[Definition~3.6]{J2} defines the stability of $C$.
We have:
\begin{thm}\l{decay order estimate}
Let $C$ be stable in the sense of Joyce.
Then there exists $b\in\R^{2m}$ such that
$W-b\in\mathcal{W}_0$.
\end{thm}
\begin{proof}
By the stability of $C$
we have $\Lambda\cap(1,2)=\emptyset$ and
$E_1=\{b\cdot x:b\in\R^{2m}\}$.
We take $\lambda\in(1,1+\ep)$.
Let $h_W$ be as above. Then by Lemma~\ref{improvement} we have
$h_W=O(r^{\lambda})$.
By Lemma~\ref{delta h}, therefore, we have
$\Delta_Ch_W=O(r^{2(\lambda-2)})$.
Applying Corollary~\ref{linearized equation} to $h_W,\lambda,2(\lambda-2)$
in place of $u,p,q$
we get $h_W=b\cdot x|_C+O(r^{2(\lambda-2)+2})$.

We may suppose that for each $t\in[0,1]$
there exist a compact subset $K_t$ of $\R^{2m}$ and
a $1$-form $w_t$ on $C\-\,\ov{\!B_R}$ such that
$(\spt\|W\|-tb)\- K_t
=\Phi_C(\gr w_t)$.
We put $\be=(b\cdot x)\circ\Phi_C$.
We have then a function $\be:UC\to\R$.
Notice that $\Phi_C^{-1}(W-b)$ is the image of the time-one map of
the flow generated by $d\be\lrcorner\om_C$. 
Then we have
$\d w_t/\d t=-w_t^*d\be$
and so $w_1=w_0-d\int_0^1w_t^*\be dt
=\pi_\Si^*\eta_W+d(h_W-\int_0^1w_t^*\be dt)$.
Hence we get
$h_{W-b}=h_W-\int_0^1w_t^*\be dt=h_W-\be|_C-
\int_0^1(w_t^*\be-\be|_C)dt=O(r^{2\lambda-2})+O(r^{\lambda-1})
=O(r^{2\lambda-2})$.
By results of Joyce~\cite[Equations~(7.7) and (7.8)]{J1}
we have
$\Delta_Ch_{W-b}=O(r^{4\lambda-8})$.
Applying Corollary~\ref{linearized equation} to $h_{W-b},2\lambda-2,4\lambda-8$
in place of $u,p,q$
we get $c\in\R$ such that
$h_{W-b}=c+O(r^{4\lambda-6})=O(r^0)$
as we may suppose $4\lambda-6<0$.
This completes the proof.
\end{proof}
Let $\lambda<2$ and $W\in\mathcal{W}_\lambda$.
Take a compact subset $K_W$ of $\R^{2m}$,
an $R>0$ and a diffeomorphism $f_W:C\-\,\ov{\!B_R}\to\spt\|W\|\- K_W$
with $f_W-i_C=O(r^{\lambda-1})$.
Then we have a Riemannian metric $f_W^*g_0$ over $C\-\,\ov{\!B_R}$.
With respect to $f_W^*g_0$ we can define the Laplacian
$\Delta_W:C^\iy(C\-\,\ov{\!B_R};\R)
\to C^\iy(C\-\,\ov{\!B_R};\R)$.
We have then:
\begin{proposition*}
If $u\in C^\iy(C\-\,\ov{\!B_R};\R)$ and $u=O(r^\al)$
then we have
$\Delta_Wu=\Delta_Cu+O(r^{\al+\lambda-4})$.
\end{proposition*}
\begin{proof}
Since $f_W-i_C=O(r^{\lambda-1})$ we get
$f_W^*g_0=i_C^*g_0+O(r^{\lambda-2})$.
We denote by $\nabla_W$ and $\nabla_C$
the Levi-Civita connexions over $C\-\,\ov{\!B_R}$ with respect to $f_W^*g_0$
and $i_C^*g_0$ respectively.
We have then $\nabla_W=\nabla_C+O(r^{\lambda-3})$
and so $\Delta_Wu=\Delta_Cu+O(r^{\al+\lambda-4})$
as we want.
\end{proof}
We give a corollary to this:
\begin{cor}\l{improvement2}
If $u\in C^\iy(C\-\,\ov{\!B_R};\R)$,
$\Delta_Wu=0$ and $u=O(r^0)$
then we have $u=c+O(r^{2-m})$ for some $c\in\R$.
\end{cor}
\begin{proof}
By the proposition above we have
$\Delta_Cu=O(r^{0+\lambda-4})=O(r^{\lambda-4})$.
If $\lambda-2<2-m$ then we can complete the proof
by applying Corollary~\ref{linearized equation} to $0,\lambda-2$
in place of $p,q$ respectively.

We suppose therefore $\lambda-2\geq2-m$.
We take $\lambda'\in(\lambda,2)$ such that $\lambda'-2>2-m$.
Applying Corollary~\ref{linearized equation} to $0,\lambda'-2$
in place of $p,q$ respectively, we get $u=c+O(r^{\lambda'-2})$
for some $c>0$.
We have thus improved the decay order estimate for $u$,
and so we can complete the proof in a way similar to the proof of
Lemma~\ref{improvement}.
\end{proof}

We suppose now
\[C=\{(z_1,\cdots,z_m)\in\C^m\-\{0\}:
|z_1|=\cdots=|z_m|,z_1\cdots z_m\in(0,\iy)\}.\]
This is an extension of \eqref{C} to dimension $m$.
Harvey and Lawson~\cite[Chapter~III.3.A, Theorem~3.1]{HL}
prove that $C$ is a special Lagrangian submanifold of $(\R^{2m}\-\{0\},
\om_0,J_0,\Om_0)$.

We define a $T^{m-1}$-action on $\C^m$ as follows.
We write $T^{m-1}=S^1\times\cdots\times S^1$ and $S^1=\{t\in\C:|t|=1\}$.
For each $j\in\{2,\cdots,m\}$
we define the $j$-th $S^1$-action on $\C^m$ by setting
$t\cdot(z_1,\cdots,z_j,\cdots,z_m)=(tz_1,\cdots,t^{-1}z_j,\cdots,z_m)$
for each $(z_1,\cdots,z_m)\in\C^m$.
This action preserves the subset $C\su\C^m$.

We also define a map $\mu_j:\C^m\to\R$ by
$2\mu_j(z_1,\cdots,z_m)=|z_1|^2-|z_j|^2$.
We shall identify $\R$ with the Lie algebra of $S^1$
so that $\mu_j$ will be the moment map on $(\C^m,\om_0)$
with respect to the $S^1$-action.
One important property is that $\mu_j|_C\=0$.

We prove:
\begin{thm*}
Let $m>2$. If $W\in\mathcal{W}_0$
and $j\in\{2,\cdots,m\}$ then
we have $\grad_{TW}\mu_j=0$
almost everywhere on $\R^{2m}$
with respect to $\|W\|$.
\end{thm*}
\begin{proof}
We can take a compact subset $K_W$ of $\R^{2m}$,
an $R>0$ and a diffeomorphism $f_W:C\-\,\ov{\!B_R}\to
\spt\|W\|\- K_W$ with $f_W-i_C=O(r^{-1})$.
Since $\mu_j=O(r^2)$ and $\mu_j\circ i_C=0$ we get
$f_W^*\mu_j=\mu_j\circ f_W=\mu_j\circ f_W-\mu_j\circ i_C
=\int_0^1d\mu_j|_{(1-t)f_W+t i_C}(f_W-i_C)dt
=O(r^{2-1}r^{-1})=O(r^0)$ and
so $f_W^*\mu_j=O(r^0)$.
By a result of Joyce~\cite[Lemma~3.4]{J2} we have
$\Delta_Wf_W^*\mu_j=0$ for each $j=2,\cdots,m$.
By Corollary~\ref{improvement2} we can find
$c_j\in\R$ such that
$f_W^*\mu_j-c_j=O(r^{2-m})$.
Putting $\mu_j'=\mu_j-c_j$ we get
$f_W^*\mu_j'=O(r^{2-m})$.
We have clearly $\grad\mu_j'=\grad\mu_j$.

Take a smooth function
$\chi:\R\to[0,1]$ with $\chi=1$ on $B_1$ and $\chi=0$ on $\R^6\-B_2$.
Let $R>0$, and define a function $\chi_R:\R^{2m}\to[0,1]$
by setting $\chi_R(x)=\chi(|x|/R)$.
Since $W$ has first variation $0$ in $B_{2R}$ 
we get then
\[\int_{B_{2R}}\di_{TW}(\chi_R\mu_j'\grad\mu_j)d\|W\|=0.\]
Also by a result of Joyce~\cite[Lemma~3.4]{J2} we have
$\di_{TW}\grad\mu_j=0$ and so
\begin{equation}\l{mu grad mu}
\int_{B_{2R}}\mu_j'(\grad\chi_R,\grad_{TW}\mu_j)
+\chi_R|{\grad_{TW}\mu_j}|^2d\|W\|=0.
\end{equation}
Notice that
$(\grad\chi_R,\grad_{TW}\mu_j)
=(d\chi_R,d\mu_j|_{\spt\|W\|\- K_W})$
on $\spt\|W\|\- K_W$.
Take $R$ sufficiently large so that
$K_W\subset B_R$.
Then we have $\grad\chi_R=0$ on $K_W$
and so
$(\grad\chi_R,\grad_{TW}\mu_j)
=(d\chi_R,d\mu_j|_{\spt\|W\|\- K_W})$
on $\spt\|W\|$.
We have therefore
\[\int_{B_{2R}}-\mu_j'(\grad\chi_R,\grad_{TW}\mu_j)d\|W\|
\leq\sup_{\spt\|W\|\- K_W}|\mu_j'(d\chi_R,d\mu_j)|
\int_{B_{2R}}d\|W\|\]
and so by \eqref{mu grad mu} we have
\begin{equation}\l{mu grad}
\int_{B_{2R}}\chi_R|{\grad_{TW}\mu_j}|^2d\|W\|
\leq\sup_{\spt\|W\|\- K_W}|\mu_j'(d\chi_R,d\mu_j)|
\int_{B_{2R}}d\|W\|.
\end{equation}
Since $f_W:C\-\,\ov{\!B_R}\to\spt\|W\|\- K_W$
is a diffeomorphism we get
\begin{equation}
\sup_{\spt\|W\|\- K_W}|\mu_j'(d\chi_R,d\mu_j)|=
\sup_{C\-\,\ov{\!B_R}}|f_W^*\mu_j'(df_W^*\chi_R,df_W^*\mu_j)|.
\end{equation}
Since $\chi_R=\chi(r/R)$
we get $|d\chi_R|\leq kR^{-1}$ for some $k>0$
independent of $R$.
Since $f_W^*\mu_j'=O(r^{2-m})$ we get
\begin{equation}\l{grad mu2}
\sup_{C\-\,\ov{\!B_R}}|f_W^*\mu_j'(d\chi_R,df_W^*\mu_j)|
\leq kR^{2-m}R^{-1}R^{1-m}=kR^{2-2m}
\end{equation}
for some $k>0$ independent of $R$.
By \eqref{mu grad}--\eqref{grad mu2} we have
\[\int_{B_{2R}}\chi_R|{\grad_{TW}\mu_j}|^2d\|W\|
\leq kR^{2-2m}
\int_{B_{2R}}d\|W\|.\]
On the other hand, by the monotonicity formula,
we have $\int_{B_{2R}}d\|W\|\leq kR^m$
for some $k>0$ depending only on $m$ and $T_\iy W$.
We have therefore
\[\int_{B_{2R}}\chi_R|{\grad_{TW}\mu_j}|^2d\|W\|\leq kR^{2-m},
k>0\text{ independent of }R.\]
Letting $R\to\iy$ we get $\int_{\R^{2m}}|{\grad_{TW}\mu_j}|^2d\|W\|=0$ because $m>2$,
completing the proof.
\end{proof}
We give a corollary to the theorem above.
We define a map $f:\C^m\to\C$ by
setting $f(z_1,\cdots,z_m)=i^{m+1}z_1\cdots z_m$
for each $(z_1,\cdots,z_m)\in\C^m$.
We put $\im f=(f-\bar{f})/2i:\C^m\to\R$
and $F=(\mu_2,\cdots,\mu_m,\im f):
\C^m\to\R^m$.
We have then:
\begin{cor}
\l{T_zV}
If $W\in\mathcal{W}_0$ then we have
$T_yW=\Ker dF|_y$
for $\|W\|$-almost every $y\in\R^{2m}$.
\end{cor}
\noindent
This follows readily from the proof of
Harvey and Lawson
\cite[Chapter~III.3.A, Theorem~3.1]{HL}.
%or a more general result of Goldstein~\cite[Theorem~A]{Goldstein}.

We have moreover:
\begin{cor}\l{F}
For every $W\in\mathcal{W}_0$ there exists $c\in\R^m$
such that $F=c$ on $\spt\|W\|$.
\end{cor}
\begin{proof}
By Corollary~\ref{T_zV} we have
$dF|_{\spt\|W\|\- K_W}=0$,
and $F|_{\spt\|W\|\- K_W}$ is therefore locally constant.
Since $\spt\|W\|\- K_W\cong C\-\,\ov{\!B_R}\cong(R,\iy)\times T^{m-1}$
we see that $\spt\|W\|\- K_W$ is connected, and
$F|_{\spt\|W\|\- K_W}$ is therefore constant;
i.e. we have $F|_{\spt\|W\|\- K_W}=c$ for some $c\in\R^m$.

Put $\phi=|F-c|^2$.
Then we have $\phi=0$ on $\spt\|W\|\- K_W$
and so $\spt\phi\cap\spt\|W\|\subset K_W$.
By Corollary~\ref{T_zV} we have $\grad_{TW}\phi=0$
almost everywhere on $\R^{2m}$ with respect to $\|W\|$.

We shall now use a result of
Michael and Simon~\cite[Theorem~2.1]{Michael Simon},
who prove a
Poincar\'e--Sobolev inequality for varifolds;
we refer also to Simon~\cite[Theorem~18.6]{Simon3},
who uses varifolds more explicitly.
We are going to use the following version:
\begin{lem*}
Let $W$ be a stationary integral varifold of dimension $m$ in $(\R^n,g_0)$.
Suppose that we have a smooth function $\phi:\R^n\to[0,\iy)$
with $\spt\phi\cap\spt\|W\|$ compact and $\grad_{TW}\phi=0$
almost everywhere on $\R^n$
with respect to $\|W\|$.
Then we have $\phi=0$ on $\spt\|W\|$.
\end{lem*}
We give a proof for the sake of clarity:
\begin{proof}
We define $r:\R^n\to[0,\iy)$
by setting $r(y)=|y|$.
We can then define
$\d_r=\d/\d r$ as a smooth vector field on $\R^n\-\{0\}$.
It is easy to see that $r\d_r$
extends smoothly to $\R^n$.
Let $\phi$ be as above, and let $\chi$ be a compactly-supported
smooth function on $\R^n$ with $\chi=1$ on $\spt\phi\cap\spt\|W\|$.
Then we can define
$\phi r\d_r$ as a smooth vector field on $\R^n$.
Since $W$ has first variation $0$ we get
\[\int_{\R^n}
\di_{TW}\chi\phi r\d_rd\|V\|=0.\]
Since $\chi=1$ on $\spt\phi\cap\spt\|W\|$ we get
$\int_{\spt\phi\cap\spt\|W\|}\di_{TW}\phi r\d_rd\|W\|=0$.
Since $\grad_{TW}\phi=0$ on $\spt\|W\|$ we get
$\int_{\spt\phi\cap\spt\|W\|}\phi\di_{TW}r\d_rd\|W\|=0$.
Since $\di_{TW}r\d_r=m$ we get
$\int_{\spt\phi\cap\spt\|W\|}m\phi d\|W\|=0$.
Since $\phi\geq0$ we get $\phi=0$ on $\spt\phi\cap\spt\|V\|$,
completing the proof.
\end{proof}
Hence we get $\phi=|F-c|^2=0$ on $\spt\|W\|$,
completing the proof of Corollary~\ref{F}.
\end{proof}

We suppose now $m=3$.
For the fibres of $F:\C^3\to\R^3$
we have an explicit description of
the topology and asymptotic behaviour at infinity,
which we shall use next.
The behaviour of $F:\C^m\to\R^m$
is rather complicated if $m>3$, which we shall not discuss.

We put $Y=\{(a,0,0)\in\R^3:a\geq0\}\cup\{
(0,a,0)\in\R^3:a\geq0\}\cup\{(-a,-a,0)\in\R^3:a\geq0\}$.
We note that if $c\in\R^3\- Y$
then $F^{-1}(c)$ has no fixed point with respect to the $T^2$-action
\[(e^{i\theta},e^{i\phi})\cdot(z_1,z_2,z_3)=(e^{i\theta}z_1,
e^{i\phi}z_2,e^{-i\theta-i\phi}z_3).\]

We have:
\begin{proposition}\l{c in Y}
Let $W\in\mathcal{W}$
and suppose $\spt\|W\|\subset F^{-1}(c)$ for some $c\in\R^3$.
Then we have $c\in Y$.
\end{proposition}

For the proof we shall use:
\begin{lem}\l{RS^1S^1}
If $c\in\R^3\- Y$
then $F^{-1}(c)$ is a submanifold of $\R^6$
diffeomorphic to $\R\times S^1\times S^1$
and asymptotic to $C\cup-C$ with multiplicity $1$ at infinity.
\end{lem}
\begin{proof}
We put $c=(c_1,c_2,c_3)$.
We may suppose $c_1\geq c_2\geq0$ without loss of generality.
Since $(c_1,c_2,c_3)\in\R^3\- Y$ we get $c_3\neq0$ or $c_2>0$.
For each $t\in\R$ we can find a unique $\phi_c(t)\in[0,\iy)$
such that
\[(\phi_c(t)+c_1)
(\phi_c(t)+c_2)\phi_c(t)=|t+ic_3|^2.\]
It is easy to see that $\phi_c(t)$ depends smoothly on $t$.
We put $\psi_c(t)=\sq{\phi_c(t)+c_1}
\sq{\phi_c(t)+c_2}$.
Since $c_3\neq0$ or $c_2>0$ we get
$\psi_c(t)>0$ for every $t\in\R$,
and so we can define $\bigl(\psi_c(t)\bigr)^{-1}>0$
for every $t\in\R$.
Define a smooth map $\Phi_c:\R\times S^1\times S^1\to F^{-1}(c)$
by setting
\[\Phi_c(t,u,v)=(\sq{\phi_c(t)+c_1}u,\sq{\phi_c(t)+c_2}v,
\frac{t+ic_3}{\psi_c(t)uv}),\]
and a smooth map $\Psi_c(t):F^{-1}(c)\to\R\times S^1\times S^1$
by setting
\[\Psi_c(z_1,z_2,z_3)=(\re z_1z_2z_3,\frac{z_1}{(\phi_c(\re z_1z_2z_3)+c_1)^{1/2}},
\frac{z_2}{(\phi_c(\re z_1z_2z_3)+c_2)^{1/2}}).\]
Then $\Psi_c\circ\Phi_c$ is clearly the identity map of
$\R\times S^1\times S^1$.
It is also easy to see that $\Phi_c\circ\Psi_c$
is the identity map of $F^{-1}(c)$,
and so $F^{-1}(c)$ is a submanifold of $\C^3$
diffeomorphic to $\R\times S^1\times S^1$.
It is also easy to see that $F^{-1}(c)$
is asymptotic to $C\cup-C$ with multiplicity $1$ at infinity.
\end{proof}

\begin{proof}[Proof of Proposition~$\ref{c in Y}$]
We give a proof by contradiction,
and so suppose $c\notin Y$.
By Lemma~\ref{RS^1S^1}, then,
$F^{-1}(c)$ will be a connected submanifold of $\R^6$.
By a constancy theorem of Allard~\cite[Theorem~4.6(3)]{Allard}
or Simon~\cite[Theorem~41.1]{Simon3},
therefore,
the multiplicity function $\Theta_W$ will be constant on $F^{-1}(c)$ (we recall that $\Th_W$ is characterized by the condition $\Th_W\cH^3=\|W\|$ where $\cH^3$ denotes the $3$-dimensional Hausdorff measure on $\R^6$).
On the other hand we have $\Theta_W=1$ near infinity on $\spt\|W\|$
and so we shall have $\Theta_W=1$ on $F^{-1}(c)$,
which implies $\|W\|=|F^{-1}(c)|$.
By Lemma~\ref{RS^1S^1}, however,
$F^{-1}(c)$ is asymptotic at infinity to $C\cup-C$
with multiplicity $1$,
which contradicts that $W$ is asymptotic at infinity to $C$
with multiplicity $1$.
This completes the proof.
\end{proof}

Suppose now that $c\in Y\subset\R^2\times\{0\}$.
We prove:
\begin{proposition}\l{fibre}
Let $W\in\mathcal{W}$ and suppose
$\spt\|W\|\subset F^{-1}(c)$ for some $c\in Y\su\R^2\t\{0\}$.
Then we have
$\|W\|=|sL|$
for some $s>0$ and $L\in\{C,L_1,L_2,L_3\}$.
\end{proposition}

\begin{proof}
We treat the two cases $c=\bm0,$ $c\ne\bm0$ individually but in both cases the main tool is the constancy theorem (which we have already used in the proof of Proposition \ref{c in Y}).

If $c=\bm0\in Y\su\R^3$ then we have $F^{-1}(\bm0)=C\cup\{\bm0\}\cup-C$. Let $U_+:=\{\re z_1z_2z_3>0\}$ and $U_-:=\{\re z_1z_2z_3<0\}$, which are open subsets of $\C^3$. Then we have
$F^{-1}(\bm0)\cap U_+=C$ and 
$F^{-1}(\bm0)\cap U_-=-C$, which are submanifolds of $U_+$ and $U_-$ respectively.
Hence by the constancy theorem we find some integers $n_{\pm}\ge0$ such that $\|W\|\llcorner U_\pm=n_\pm|F^{-1}(0)\cap U_\pm|$ where $\|W\|\llcorner U_\pm$ denotes the Radon measure on $\R^6$ with $\|W\|\llcorner U_\pm(A)=\|W\|(U_\pm\cap A)$ for $A\su\R^6$. On the other hand $W$ is asymptotic at infinity to $C$, and so $n_+=1$ and $n_-=0$, which implies $\|W\|=|C|$ as we want.

We turn now to the case $c\ne\bm0.$ One easily sees that $F^{-1}(c)$ may be written as the union of two $S^1\t\R^2$ intersecting at $S^1\t\{\bm0\}$ and $F^{-1}(c)\cap U_+$, $F^{-1}(c)\cap U_-$ are diffeomorphic to $S^1\t(\R^2\-\{\bm0\}).$ Thus $F^{-1}(c)$ is topologically different from $F^{-1}(\bm0)$, but otherwise the treatment above is valid with $F^{-1}(c)$ in place of $F^{-1}(\bm0)$, which implies $\|W\|=|F^{-1}(c)\cap U_+|$.
Finally from \eq{L1}--\eq{L3}, the definition of $L_1,L_2,L_3,$ it follows readily that the closure of $F^{-1}(c)\cap U_+$ in $\R^6$ is equal to $sL$ for some $s>0$ and $L\in\{L_1,L_2,L_3\}$, which completes the proof of Proposition \ref{fibre}.
\end{proof}

We shall now complete the proof of Theorem~\ref{uniqueness}.
We suppose that $W$ is as in Theorem~\ref{uniqueness}.
By Theorem~\ref{decay order estimate}, then,
we can find $b\in\R^6$ such that $W-b\in\mathcal{W}_0$.
By Corollary \ref{F}, therefore,
we can find $c\in\R^3$
such that $\spt\|W-b\|\subset F^{-1}(c)$.
By Proposition~\ref{c in Y}, therefore, we have $c\in Y$.
By Proposition~\ref{fibre}, therefore, we have
$\|W-b\|=|sL|$ for some $s>0$ and $L\in\{C,L_1,L_2,L_3\}$
as we want.

%We begin with a review of the theory of
%Joyce~\cite{J1,J2,J3,J4,J5}.
%As in \S\ref{intro} we define $C$ by \eqref{C}.
%We put $\Si=C\cap S^5$.
%Let $u\in C^\iy(C;\R)$
%and $\al\in\R$.
%we shall write $u=O(r^\al)$
%as $r\to0$ if we have
%\[\lim_{\ep\to+0}\sup_{C\cap B_\ep}|r^{-\al+k}\nabla^ku|<\iy
%\text{ for every }k=0,1,2,3,\cdots\]
%where $\nabla$ denotes the Levi-Civita connexion
%over the Riemannian manifold $(C,i_C^*g_0)$,
%and $|\bullet|$ denotes the norm with respect to $i_C^*g_0$.
%Then we shall write $u=O(r^\al)$ as $r\to0$ if
%there exists $\de>0$ such that
%\[\sup_{(0,\de)\times\Si}|r^{-\al+k}\nabla^ku|<\iy
%\text{ for every }k=0,1,2,3,\cdots\]
%and write $u=O(r^\al)$ as $r\to0$ if
%there exists $R>0$ such that
%\[\sup_{(R,\iy)\times\Si}|r^{-\al+k}\nabla^ku|<\iy
%\text{ for every }k=0,1,2,3,\cdots\]
%where $\nabla$ denotes the Levi-Civita connexion induced from $g_0$,
%and $|\bullet|$ denotes the norm induced from $g_0$.

\section{Combining Results of \S\S\ref{bu}--\ref{un}}\l{bg1}
As in \S\ref{intro} the proof of Theorems \ref{main2} and \ref{main3}, the main theorems of the present paper, is based on the results of \S\S\ref{bu}--\ref{un} above.
In the present section therefore we combine those results into a convenient form (Theorem \ref{BG} below).
We begin by recalling some basic notation from \S\ref{intro}.

Let $(M,\om,J,\Om)$ be an almost Calabi--Yau manifold of complex dimension $3$.
Let $\cV$ denote the space of all compactly-supported special Lagrangian integral varifolds with no boundary in $M$.

Let $\cX$ denote the subspace of $\cV$ consisting of compact special Lagrangian $3$-folds in $M$ with only one singular point modelled on $C$ with multiplicity $1$ where $C$ denotes the $T^2$-cone in $\C^3$ given by \eq{C}.

Let $L_1,L_2,L_3$ be as in \eq{L1}--\eq{L3}, which are non-compact special Lagrangian submanifolds properly-embedded in $\C^3$ and asymptotic at infinity to $C$ with multiplicity $1$.
For each $s>0$ and $L\in\{L_1,L_2,L_3\}$ we write $sL:=\{s\bm z\in\C^3:\bm z\in L\}$.

In the notation above the main result of the present section may be stated briefly as follows:
\begin{thm}\l{bg}
Let $X\in\cX$ and let $x$ denote the unique singular point of $X$.
Then there exist a neighbourhood $B_x$ of $x$ in $X$ and a neighbourhood $\cU$ of $X$ in $\cV$ such that if $V\in\cU\-\cX$ then the following two statements hold:
\iz
\item[\bf(I)]
$V$ is a multiplicity-one non-singular varifold and $V$ restricted to $M\-B_x$ is $C^1$-close to $X;$
\item[\bf(II)]
there exists an open set $B_{x,V}\su B_x$ with the following two properties:
{\bf(i)} $V$ restricted to $B_{x,V}$ is $C^1$-close to $sL$ for some $s>0$ and $L\in\{L_1,L_2,L_3\};$
{\bf(ii)}
$V$ restricted to $B_x\-B_{x,V}$ is $C^1$-close to the tangent cone to $X$ at $x$.
\iz
\end{thm}
\begin{proof}
By Allard's regularity theorem we can take a neighbourhood $\cU$ of $X$ in $\cV$ and a neighbourhood $B_x$ of $x$ such that if $V\in\cU\-\cX$ then $V$ will satisfy the latter part of (I) above.

Near $x$ we can apply Theorem \ref{bubble-off} so that as $V$ approaches $X$ we can take its re-scaled limit $W$ with positive energy;
more precisely the proof of Theorem \ref{bubble-off} implies that we can take $y\in B_x,$ $\de>0$ and a special Lagrangian integral varifold $W$ with no boundary in $\C^3$ satisfying the following two properties:
\iz
\item[(a)]
$\cE\bigl(V\llcorner(B_x\-B_\de(y)\bigr)\le\ep/2$ where $B_\de(y)$ denotes a $\de$-ball about $y$ in $M$ and $\ep$ denotes a constant given by Theorem \ref{AAS};
\item[(b)]
if we identify $B_x$ with an open set in $\C^3,$ translate $V$ by $y\in\C^3$ and dilate it by $\de^{-1}$ then the resulting varifold $\de^{-1}(V-y)$ is close to $W$ in the varifold topology.
\iz
From (b) and Theorem \ref{uniqueness} it follows that $W$ is a multiplicity-one non-singular varifold represented by some $L\in\{L_1,L_2,L_3\}$ up to dilation and translation in $\R^6\cong T_xM$.
This combined with Allard's regularity theorem implies (i) above.

We have not seen yet what happens to $V$ in the annular region $B_x\-B_\de(y)$ but in that region $V$ has little energy as in (a) above.
This combined with Theorem \ref{AAS} implies (ii) above with $B_{x,V}=B_\de(y).$

Finally the former part of (I) follows from (i), (ii) and the latter part of (I).
\end{proof}

In Theorem \ref{bg} above we have frequently used the expression `$C^1$-close' for Lagrangian submanifolds, but in later sections we shall need to describe them in Weinstein neighbourhoods \cite{W} (a class of tubular neighbourhoods of Lagrangian submanifolds).
In what follows therefore we introduce some notation concerning Weinstein neighbourhoods and re-state Theorem \ref{bg} in that notation.

We take a linear isomorphism
$\ga:\R^6\to T_xM$
with $\ga^*g|_x=g_0$, $\ga^*J|_x=J|_0$,
$\ga^*\Om|_x=\Om_0$
and $\ga(C)$ a multiplicity $1$ smooth tangent cone
to $X$ at $x$.
We write $X':=X\-\{x\}$.

As in \S\ref{SL section} we define $\psi:M\to(0,\iy)$ by \eqref{psi}.
We have then $\om_0=\psi^2(x)\ga^*\om|_x$.
By Darboux's theorem we can find a real number $\de>0$ and an embedding $\Ga:B_\de\to M$ with $\Ga(0)=x$, $d\Ga|_0=\psi(x)\ga$ and $\Ga^*\om|_x=\om_0$.

We define $T^*C$, $\om_C$, $U_C$ and $\Phi_C$ as in Lemma~\ref{UC}.
Joyce \cite[Theorem~4.4 and Lemma~4.5]{J1} proves that making $\de>0$ smaller if necessary we can take an embedding $f_X$ of $C\cap B_\de$ into $X'$, a function $h_X:C\cap B_\de\to\R$ and an $\al>2$ with $h_X=O(r^\al)$ such that $f_X^{-1}(X')=\Phi_C(\gr dh_X)$.
We put $Z=X'\-f_X(C\cap\,\ov{\!B_\de})$.
It is clear that $Z$ is an open subset of $X'$ with boundary diffeomorphic to $T^2$.

We denote by $T^*X'$ the cotangent bundle over $X'$, by $0_C$ the zero-section of $T^*X'$, and by $\om_X$ the canonical symplectic form on $T^*X'$. Since $f_X$ maps $C\cap B_\de$ diffeomorphically onto $X'\-Z$ we get a vector-bundle isomorphism $T^*f_X:T^*(X'\-Z)\to T^*(C\cap B_\de)$
covering $f_X^{-1}:X'\- Z\to C\cap B_\de$.
Joyce~\cite[Theorem~4.6]{J1} constructs a neighbourhood $UX'$ of $\im0_X$ in $T^*X'$ and a diffeomorphism $\Phi_X$ of $UX'$ into $M$ with $\Phi_X\circ0_X=i_X$, $\Phi_X^*\om=\om_X$ and
\[\Phi_X|_{UX'\-T^*Z}
=\Ga\circ\Phi_C\circ(+dh_X)\circ T^*f_X\]
where $+dh_X$ denotes the fibrewise translation of $T^*C$
by $dh_X$.

For each $y\in M$
we denote by $P|_y$ the set of all
linear isomorphisms $\phi:\R^6\to T_yM$
with $\phi^*g|_y=g_0$,
$\phi^*J|_y=J_0$ and $\phi^*\Om|_y=\Om_0$.
We put $P=\bigcup_{y\in M}P|_y$.
It is clear that $P$ is a principal bundle over $M$
with structure group $SU_3$.
We have $(x,\ga)\in P$.
By a result of Joyce~\cite[Theorem~5.2]{J2}
we can take a neighbourhood $U_{x,\ga}$ of $(x,\ga)$ in $P$
such that for all $p=(y,\phi)\in U_{x,\ga}$
we can construct
embeddings $\Ga_p$ of $B_\de$ into $M$
depending smoothly on $p$ with
$\Ga_p(0)=y$, $d\Ga_p|_0=\phi$ and
$\psi^2(y)\Ga_p^*\om=\om_0$, and
embeddings $\Phi_p$ of $UX'$ into $M$
depending smoothly on $p$ with
$\Phi_p\circ0_X=i_X$, $\Phi_p^*\om=\om_X$ and
\[\Phi_p|_{UX'\- T^*Z}=\Ga_p\circ\Phi_C\circ(+dh_X)\circ T^*f_X.\]

We turn now to $L\in\{L_1,L_2,L_3\}$.
We take an open subset $K$ of $L$ with boundary diffeomorphic to $T^2$, a real number $R>0$ and a diffeomorphism $f_L$ of $C\-\,\ov{\!B_R}$ onto
$L\-\,\ov{\!K}$ with
$f_L-i_C=O(r^{-1})$.
Making $R>0$ larger if necessary
we can find a $1$-form $\eta_L$ on $C\-\,\ov{\!B_R}$
with $\Phi_C(\gr \eta_L)=L\- K$.

We denote by $T^*L$ the cotangent bundle over $L$,
by $0_L$ the zero-section of $T^*L$, and by $\om_L$
the canonical symplectic form on $T^*L$.
Since $f_L$ maps $C\-\,\ov{\!B_R}$ diffeomorphically onto $L\- K$
we get a vector-bundle isomorphism
$T^*f_L:T^*(L\- {K})\to T^*(C\-\,\ov{\!B_R})$
covering $f_L^{-1}:L\- K\to C\cap\,\ov{\!B_R}$.
Joyce \cite[Theorem 7.5]{J1} constructs a neighbourhood $UL$ of $\im0_L$ in $T^*L$ and a diffeomorphism $\Phi_L$ of $UL$ into $\R^6$ with $\Phi_L^*\om_0=\om_L$, $\Phi_L\circ0_L=i_L$ and
\e\l{eta}\Phi_L|_{UL\-T^*{K}}
=\Phi_C\circ(+\eta_L)\circ T^*f_L\e
where $+\eta_L$ denotes the fibrewise translation of $T^*C$ by $\eta_L$.

We are ready now to refine the statement of Theorem \ref{bg}:
\begin{thm}\l{BG}
For each $\ep>0$ there exists a neighbourhood $\cU$ of $X$ in $\cV$ such that if $V\in\cU\-\cX$ then $V$ will be a multiplicity-one non-singular varifold in $M$ and we can find some $p\in P,$ $s>0,$ $L\in\{L_1,L_2,L_3\},$ a closed $1$-form $\be_C$ on $C\cap A_{sR,\de}$ with $|\be_C|^{1,\cyl}_{[sR,\de]}<\ep,$ a closed $1$-form $\be_L$ on $\Kh$ with $|\be_L|_{C^1(\widehat{K})}<\ep,$ and a closed $1$-form $\be_X$ on $\widehat{Z}$ with $|\be_X|_{C^1(\widehat{Z})}<\ep$ such that on $f_L^{-1}(\widehat{K})\cap A_{sR,\de}$ we have $s^2(f_L^*\be_L+\eta_L)=\be_C,$ on $f_X^{-1}(\widehat{Z})\cap A_{sR,\de}$ we have $f_X^*\be_X+dh_X=\be_C,$ and
\e\l{V}V=\Ga_p\bigl(s\Phi_L(\gr\be_L)\bigr)
\cup\Ga_p\circ\Phi_C(\gr\be_C)
\cup\Phi_X(\gr\be_X).\e
\end{thm}
\begin{rem*}
Here and subsequently if $V$ is a multiplicity-one varifold in $M$ then we shall treat $V$ as a subset of $M$ to simplify the notation as in \eq{V}.

The right-hand side of \eq{V} defines a submanifold of $M$ because of the two conditions $s^2(f_L^*\be_L+\eta_L)=\be_C$ 
on $f_L^{-1}(K)\cap A_{sR,\de}$ and
$f_X^*\be_X+dh_X=\be_C$
on $f_X^{-1}(Z)\cap A_{sR,\de}$.
\end{rem*}

\section{Remarks on Joyce's Work}
This section will be devoted to several remarks on Joyce's work \cite{J1,J2,J3,J4,J5}.

We begin in \S\ref{cd} by recalling Joyce's topological condition \cite[Theorem 10.4]{J5} which is necessary for his gluing construction and also for the precise statement of Theorems \ref{main2} and \ref{main3}, the main theorems of the present paper.

In \S\ref{pr2} we prove Theorem \ref{main2}, which claims that if there is no $L\in\{L_1,L_2,L_3\}$ satisfying Joyce's topological condition then $X$ is {\it unsmoothable}, which is a corollary to Theorem \ref{BG}.

On the other hand Theorem \ref{main3} supposes that there is some $L\in\{L_1,L_2,L_3\}$ satisfying Joyce' topological condition and claims that there exists a neighbourhood $\cU$ of $X$ in $\cV$ such that the elements of $\cU\-\cX$ may be obtained by Joyce's gluing construction.
Its proof will be given in \S\ref{pr3} below after the preparation in \S\S\ref{Top X}--\ref{gluing}.

In \S\ref{Top X} we prove that the varifold topology on $\cX$ induced from $\cV$ is equal to a stronger topology defined by Joyce \cite[Definition 5.6]{J2}.
This result may be of independent interest as an improvement of Joyce's work in the second paper \cite{J2}.

In \S\ref{st C} we recall from Joyce's second paper \cite[Corollary 6.11]{J2} that $\cX$ is a {\it manifold} of finite dimension, which is a consequence of the fact that the $T^2$-cone $C$ is {\it stable} in the sense of Joyce \cite[\S3.2]{J2}. Our presentation will be slightly different from that of Joyce, but it is superficial and we shall only re-phrase Joyce's statement \cite[Theorem 6.10]{J2} so that we may use it directly in \S\ref{pr3}.

In \S\ref{gluing} we give an explicit description of the gluing map $G:[0,\tau)\t\cY\to\cV$ given in \S\ref{intro}.

\subsection{Joyce's Topological Condition}\l{cd}
Our main results concern the topological condition given by Joyce \cite[Theorem 10.4]{J5} which we therefore recall now.
Let $f_X:C\cap B_\de\to X'$ be as in \S\ref{bg1} above, which induces a linear map
\e f_X^*:H^1(X';\R)\to H^1(C\cap B_\de;\R)\cong H^1(T^2;\R)\e
between cohomology groups.
Its image will be denoted by $\im f_X^*$, which is a linear subspace of $H^1(T^2;\R)$.
Now we use:
\begin{lem*}[Joyce {\cite[Lemma 10.1]{J5}}]
{\rm
Let $X^\dag$ be a compact orientable $3$-dimensional manifold with boundary, and consider the natural restriction map $r^\dag:H^1(X^\dag;\R)\to H^1(\d X^\dag;\R)$ and its image $\im r^\dag$ in $H^1(\d X^\dag;\R)$.
Then we have
\[\dim_\R\im r^\dag=\ts\ha\dim_\R H^1(\d X^\dag;\R)\]
where $\dim_\R H^1(\d X^\dag;\R)$ is always an even integer as $\d X^\dag$ is a compact orientable $2$-dimensional manifold.
}
\end{lem*}
In our case $X'$ may be retracted to some $X^\dag$ with $\d X^\dag\cong T^2$ in the notation above, and so
\begin{cor*}
$\dim_\R\im f_X^*=\ts\ha\dim_\R H^1(T^2;\R)=1$.
\end{cor*}

On the other hand let $L\in\{L_1,L_2,L_3\}$ and define a closed $1$-form $\eta_L$ as in \S\ref{eta}.
Then we can define its de Rham cohomology class
\e Y(L):=[\eta_L]\in H^1(C\-\,\ov{\!B_R};\R)\cong H^1(T^2;\R)\e
which is compatible with the notation of Joyce \cite[Definition 6.2 (see also Theorem 6.6)]{J5}.
Joyce {\cite[Equation (77)]{J5}} proves indeed:
\begin{lem}\l{idp}
{\rm
$Y(L)\ne0$ for each $L\in\{L_1,L_2,L_3\}$ and moreover any two of $\{Y(L_1),Y(L_2),Y(L_3)\}$ are {\it linearly independent} as vectors in $H^1(T^2;\R)$.
}
\end{lem}

Now let $\lg Y(L)\rg$ denote the $1$-dimensional linear subspace of $H^1(T^2;\R)$ generated by $Y(L)$.
Then Joyce's topological condition \cite[Theorem 10.4 (see also Theorem 7.3)]{J5} is equivalent to:
\begin{cdt}\l{JC}
$\lg Y(L)\rg=\im f_X^*\su H^1(T^2;\R)$.
\end{cdt}

Thus there are four lines $\lg Y(L_1)\rg,\lg Y(L_2)\rg,\lg Y(L_3)\rg$ and $\im f_X^*$ in the plane $H^1(T^2;\R)\cong\R^2$, all passing thorough the origin $0\in H^1(T^2;\R)$.
The three lines $\lg Y(L_1)\rg,\lg Y(L_2)\rg,\lg Y(L_3)\rg$ are all distinct by Lemma \ref{idp}, and Condition \ref{JC} is equivalent to the fourth line $\im f_X^*$ overlapping one of those three lines.

It is now clear that there exists {\it at most} one $L\in\{L_1,L_2,L_3\}$ satisfying Condition \ref{JC}, which we have mentioned in Lemma \ref{j}.

\subsection{Proof of Theorem \ref{main2}}\l{pr2}
We are ready now to prove Theorem \ref{main2} as a corollary to Theorem \ref{BG}:
\begin{cor}\l{gluing condition}
Suppose that for every neighbourhood $\cU$ of $X$ in $\cV$
we have $\cU\-\mathcal{X}\neq\emptyset$.
Then there exists $L\in\{L_1,L_2,L_3\}$
with $\lg Y(L)\rg=\im f_X^*$.
\end{cor}
\begin{rem*}
This statement is the contraposition of Theorem \ref{main2} and so they are equivalent.
\end{rem*}
\begin{proof}[Proof of Corollary $\ref{gluing condition}$]
Theorem~\ref{BG} and the hypothesis above imply that for each $\ep>0$ there exist $L\in\{L_1,L_2,L_3\}$, a real number $s>0$, a closed $1$-form $\be_L$ with $|\be_L|_{C^1(\Kh)}<\ep$, and a closed $1$-form $\be_X$ on $\Kh,\Xh$ respectively such that
\[s^2[f_L^*\be_L+\eta_L]=[f_X^*\be_X+dh_X]\in H^1(T^2;\R)\]
and so $[f_L^*\be_L]+Y(L)\in\im f_X^*$.
Here $L$ depends on $\ep$ but letting $\ep=1/n,\,n=1,2,3,\cdots,$ and taking a subsequence we can make $L$ independent of $n$ and satisfying the following property: for infinitely many $n$ there exist $1$-forms $\be_{L,n}$ on $\Kh$ with $|\be_{L,n}|_{C^1(\Kh)}<1/n$ and $[f_L^*\be_{L,n}]+Y(L)\in\im f_X^*$.
Consequently letting $n\to\iy$ we get $Y(L)\in\im f_X^*$.

On the other hand $\lg Y(L)\rg$ and $\im f_X^*$ are both $1$-dimensional as in \S\ref{cd}, and so $\lg Y(L)\rg=\im f_X^*$, which completes the proof of Corollary \ref{gluing condition}.
\end{proof}

We can also strengthen Theorem \ref{BG} as follows: 
\begin{cor}\l{BG'}
{\rm
If there exists $L\in\{L_1,L_2,L_3\}$ with $\lg Y(L)\rg=\im f_X^*$
then such an $L$ is unique as proven in \S\ref{cd} and
{\it the statement of Theorem $\ref{BG}$ holds for that unique $L$ instead of assigning $L_1,L_2$ or $L_3$ to each $V\in\cU\-\cX$}.
}
\end{cor}
\begin{proof}
In a way similar to the proof of Corollary \ref{gluing condition} it follows that in the statement of Theorem $\ref{BG}$ we may suppose that $\lg Y(L)\rg$ is $\ep$-close to $\im f_X^*$ but such an $L$ is unique as in \S\ref{cd}, which completes the proof of Corollary \ref{BG'}.
\end{proof}

\subsection{Topology on $\cX$}\l{Top X}
%\begin{rem*}
%The map $f_X^*:H^1(X';\R)\to H^1(T^2;\R)$ may be identified with a natural homomorphism $H^1(Z^\ci;\R)\to H^1(\d Z;\R)$ where $Z^\ci$ denotes the interior of $Z$ (diffeomorphic to $X'$) and $\d Z$ denotes the boundary of $Z$ (diffemorphic to $T^2$).
%Under that identification all we need to prove Lemma \ref{one} is that $Z$ is a compact oriented $3$-dimensional manifold with boundary, and the proof contains only well-known facts about cohomology groups.
%\end{rem*}
We have used so far the varifold topology induced from $\cV$ but in what follows we prove that it is actually equal to a stronger topology given by Joyce \cite[Definition 5.6]{J2}.
We begin with:
\begin{lem}\l{zeta}
There exists a neighbourhood $\cY$ of $X$ in $\cX$ such that for each $Y\in\cY$ there exist unique $p(Y)=(y,\phi)\in P$ and a closed $1$-form $\eta_Y$ on $X'$ with
\[Y=\Phi_{p(Y)}(\gr\eta_Y)\]
where $P,\Phi_{p(Y)}$ are as in \S\ref{bg1};
moreover near the singular point $y$ in $Y$ we may write $\eta_Y=dh_Y$ for some unique function $h_Y:X'\to\R$ decaying with any rate $<3$.
\end{lem}
\begin{rem}\l{rate}
If $X$ were not modelled on the $T^2$-cone $C$ then the rate of $h_Y$ would be confined as in the definition of Joyce \cite[Definition 3.7]{J2}.
\end{rem}
\begin{proof}[Proof of Lemma $\ref{zeta}$]
By Allard's regularity theorem we may suppose that $Y$ is $C^1$-close to $X$ outside a neighbourhood of $x$ in $X$. By Theorem \ref{X top} we may also suppose that the singular point $y$ of $Y$ is close to $x$ and the tangent cone to $Y$ at $y$ is close to that to $X$ at $x$.
These two facts readily imply the existence of $p(Y)=(y,\phi),\eta_Y$ and the decay property of $\eta_Y$, as proven by Joyce \cite[Theorem 5.3]{J2}.
The uniqueness of $p(Y)=(y,\phi)$ follows from the definition of $P$, and the uniqueness of $\eta_Y$ follows from the fact that $\Phi_{p(Y)}^{-1}(Y\-\{y\})$ is $C^1$-close to $X':=X\-\{x\}$, which completes the proof of Lemma \ref{zeta}.
\end{proof}
Joyce \cite[Definition 5.6]{J2} defines a topology on $\cX$ dependent on $\mu\in(2,3)$ in which $Y\in\cX$ is close to $X$ if $p(Y)=(y,\phi)$ is close to $(x,\ga)$ in $P$ and $\eta_Y$ is small in the weighted $C^1$-space $C^1_\mu(T^*X')$ where $p(Y)=(y,\phi)$ and $\eta_Y$ are as in Lemma \ref{zeta} above. We prove:
\begin{thm}\l{same}
Joyce's topology is independent of $\mu$ and moreover equal to the varifold topology induced from $\cV$.
\end{thm}
\begin{proof}
It is clear that if $X,Y\in\cX$ are close in the $\mu$-topology then so they are in the varifold topology.
Conversely let $X,Y\in\cX$ be close in the varifold topology.
Allard's regularity theorem implies that $Y$ is $C^1$-close to $X$ outside a neighbourhood of $x$ in $X$ where $x$ denotes the singular point of $X$.
We may also write $Y$ as in Lemma \ref{zeta} and Theorem \ref{X top} implies that $p(Y)=(y,\phi)$ is close to $(x,\ga)$ in $P$.
The last part of Lemma \ref{zeta} implies that $Y$ approaches $X$ at $x$ with any rate $\mu<3$, which completes the proof of Theorem \ref{same}.
\end{proof}

Theorem \ref{same} readily extends to compact special Lagrangian $m$-folds of $X$ with isolated conical singularities in the sense of Joyce \cite{J1,J2,J3,J4,J5} where $m$ is an arbitrary integer $>2$ and the tangent cones to $X$ need not be modelled on the stable $T^2$-cone $C.$

\subsection{Consequence of Stability of $C$}\l{st C}
In \S\ref{pr3} we use the fact that $X$ is modelled on the {\it stable} cone $C$ in the sense of Joyce \cite[\S3.2]{J2} who proves indeed (in Corollary 6.11 in the same paper) that a neighbourhood of $X$ in $\cX$ is a {\it manifold} of finite dimension and its tangent space $T_X\cX$ is isomorphic to the compactly-supported de Rham cohomology group $H^1_c(X';\R)\su H^1(X';\R)$.

Here $H^1_c(X';\R)$ is embedded in $H^1(X';\R)$ because the map $f_X^*:H^0(X';\R)\to H^0(T^2;\R)$ is surjective.
It is unnecessary for our purpose but may be worth remarking that Joyce \cite{J2} deals with a more general case where $X$ has two or more singular points, and then $T_X\cX$ is isomorphic to the image of the canonical map $H^1_c(X';\R)\to H^1(X;\R)$ which need not be injective.

We shall need in \S\ref{pr3} a more detailed statement of Joyce's result above, which we therefore recall now.
We take a neighbourhood $\cY$ of $X$ in $\cX$ and define a map $H:\cY\to H^1_c(X';\R)$ as follows.
Let $\cY$ be as in Lemma \ref{zeta} so that for each $Y\in\cY$ there exist unique $p(Y)=(y,\phi)\in P$ and a closed $1$-form $\eta_Y$ on $X'$ with
\[Y=\Phi_{p(Y)}(\gr\eta_Y).\]
Consider the de Rham cohomology class $[\eta_Y]\in H^1(X';\R)$, which maps to $0$ under $f_X^*$ because of the last property in Lemma \ref{zeta}.
Consequently $[\eta_Y]$ lies in $H^1_c(X';\R)$, and we set $H(Y)=[\eta_Y]$.

With the notation above we can state the detailed version of Joyce's result \cite[Theorem 6.10]{J2}:
\begin{thm}\l{H}
{\rm
$H$ maps a neighbourhood of $X$ in $\cX$ {\it homeomorphically} onto a neighbourhood of $0$ in $H^1_c(X';\R)$.
}
\end{thm}

\subsection{Explicit Description of Gluing Map}\l{gluing}
Finally we give an explicit description of $G:[0,\tau)\t\cY\to\cV$.
The basic notation is already given in \S\ref{bg1} but we shall need some more notation.

We denote by $\eta_L'$ the harmonic $1$-form on $T^2$ with $[\eta_L']=Y(L)\in H^1(T^2;\R)$ and define a $1$-form $\etati_L$ on $C$ by setting $\etati_L:=\pi^*\eta_L'$ where $\pi$ denotes the projection of $C\cong (0,\iy)\t T^2$ onto $T^2$.

Since $\lg Y(L)\rg=\im f_X^*$ it is clear that there exists a closed $1$-form $\xi_L$ on $X'$ with $f_X^*[\xi_L]=Y(L)\in H^1(T^2;\R)$.
Following the proof of Joyce (sketched in \cite[Theorem 7.3]{J5} and completed in \cite[Theorem 6.12]{J4}) we find that for a suitable choice of $\xi_L$ the following statement holds:
\begin{thm}\l{zeta gluing}
For all $(t,Y)\in(0,\tau)\t\cY$ there exist three smooth functions $u_L:K\to\R,\,u_C:C\cap A_{tR,\de}\to\R,\,u_X:Z\to\R$
such that letting $p(Y),\eta_Y$ be as in Lemma $\ref{zeta}$ we can construct a compact special Lagrangian submanifold of $(M,\om,J,\Om)$ of the form
\[
G(t,Y):=\Ga_{p(Y)}\bigl(t\Phi_L(\gr du_L)
\cup\Phi_C(\gr t^2\etati_L+du_C)\bigr)\cup\Phi_X(\gr t^2\xi_L+\eta_Y+du_X).
\]
Moreover $|u_L|_{C^1(K)},$ $|u_C|_{[tR,\de]}^{1,\cyl},$ $|u_X|_{C^1(Z)}$ tend to $0$ uniformly in $Y$ as $t\to+0$.
\end{thm}

\section{Proof of Theorem \ref{main3}}\l{pr3}
As in \S\ref{gluing} above we suppose that there is some $L\in\{L_1,L_2,L_3\}$ satisfying Joyce's topological condition and define the gluing map $G:[0,\tau)\t\cY\to\cV$.
Theorem \ref{main3} claims that for $\tau$ and $Y$ small enough $G$ is a {\it homeomorphism} onto a neighbourhood of $X$ in $\cV$.

To prove it we shall construct an inverse map to $G$.
We shall indeed take a neighbourhood $\cU$ of $X$ in $\cV$, then define a map $F_1:\cU\to[0,\iy)$, then define a map $F_2:\cU\to\cX$ and then set $F=(F_1,F_2)$, to get a map $F:\cU\to[0,\iy)\t\cX$ inverse to $G$.

In Donaldson's situation \cite{D} (explained also by Freed and Uhlenbeck \cite{FU}) in Yang--Mills gauge theory there is a map similar to $F:\cU\to[0,\iy)\t\cX$ and the elements of $\cU$ play the r\^ole of `concentrated' instantons. There is a map similar to $F_2:\cU\to\cX$ which assigns the `centres' about which those instantons are concentrated, and there is also a map similar to $F_1:\cU\to[0,\iy)$ which measures how concentrated they are. 

We return now to our situation.
By Allard's regularity theorem there exists a neighbourhood $\cU$ of $X$ in $\cV$ such that for each $V\in\cU$ we can take a unique closed $1$-form $\be_X$ on $Z$ such that
\e\l{V-}V\cap\Phi_X(UX'\cap T^*Z)=\Phi_X(\gr\be_X)\e
where by Theorem \ref{BG} we may suppose that $V$ is a {\it multiplicity-one} varifold (with singularity at most one point) and so we may treat $V$ as a subset of $M$.

Consider the de Rham cohomology class $[\be_X]\in H^1(Z;\R)\cong H^1(X';\R)$. Then by Condition \ref{JC} we may write $f_X^*[\be_X]=rY(L)$ for some unique $r\in\R$. 

We claim that $r\ge0$.
If $V\in\cU\cap\cX$ then near $x$ we may write $V$ as the graph of an {\it exact} $1$-form as in Lemma \ref{zeta} and so $f_X^*[\be_X]=0$ in the notation above, which implies $r=0$.

If $V\in\cU\-\cX$ then by Corollary \ref{BG'} there exist a real number $s>0$, a closed $1$-form $\be_X$ on $Z$, and a closed $1$-form $\be_L$ on $K$ such that
$[f_X^*\be_X]=s^2\bigl([f_L^*\be_L]+Y(L)\bigr)\in H^1(T^2;\R)$.
Hence recalling that $\im f_X^*=\lg Y(L)\rg$ we find $[f_L^*\be_L]=cY(L)$ for some $c\in\R$.
Making $\cU$ smaller if necessary we may suppose that $\be_L$ is so small that $|c|<\frac{1}{100}$ and then $r=s^2(1+c)>0$ as we want.

Thus we have defined a map $F_1:\cU\to[0,\iy)$ and we turn now to the definition of $F_2:\cU\to\cX$.

In Theorem \ref{zeta gluing} we have taken a closed $1$-form $\xi_L$ on $X'$ with $f_X^*[\xi_L]=Y(L)$ and so $[\be_X-r\xi_L]\in\Ker f_X^*=H^1_c(X';\R)$.
We may suppose that $[\be_X-r\xi_L]$ is so small that by Theorem \ref{H} we can define $F_2(V):=H^{-1}[\be_X-r\xi_L]\in\cY$.

Thus we have defined a map $F=(F_1,F_2):\cU\to[0,\iy)\t\cX$, which readily satisfies the following two properties:
\iz
\item[(i)]
$F(Y)=(0,Y)$ for $Y\in\cU\cap\cX$;
\item[(ii)]
$F\ci G$ is the identity map of a neighbourhood of $(0,X)$ in $[0,\iy)\t\cX$.
\iz
We can also prove:
\begin{lem}\l{vC}
{\rm
There exists a neighbourhood $\cU$ of $X$ in $\cV$ such that if $V\in\cU\-\cX$ then $V$ may be written as the graph of an {\it exact} $1$-form $\al$ on $G\ci F(V)$
in a Weinstein neighbourhood of $G\ci F(V)$ in the symplectic manifold $(M,\om)$.
}
\end{lem}
\begin{rem}\l{VC}
The definitions of $F$ and $G$ readily imply that $V$ and $G\ci F(V)$ are {\it $C^1$-close} to each other. On the other hand $V$ and $G\ci F(V)$ are both {\it Lagrangian} submanifolds of $(M,\om)$.
Consequently $V$ may be written as the graph of a {\it closed} $1$-form $\al$ on $G\ci F(V)$.
Thus the substantial part of Lemma \ref{vC} is the exactness of the $1$-form $\al$.
\end{rem}
\begin{proof}[Proof of Lemma $\ref{vC}$]
We proceed in two steps:
the first step introduces some notation to define $\al;$ and the second step proves that $\al$ is exact.
\subsubsection*{First Step}
Let $\cU$ be a neighbourhood of $X$ in $\cV$, let $V\in\cU\-\cX$ and write $F(V)=(t,Y)\in(0,\iy)\t\cX$.
Making $\cU$ smaller if necessary we may apply Lemma \ref{zeta} to $Y$ and in particular we can define $p:=p(Y)\in P$ as in Lemma \ref{zeta}.

Making $\cU$ smaller if necessary we may also apply Corollary \ref{BG'} to $V$.
It is also easy to see that \eq{V} holds with $t$ in place of $s$ and with $p$ as above; i.e. there exist a closed $1$-form $\be_X$ on $Z$, a closed $1$-form $\be_C$ on $C\cap A_{tR,\de}$, and a closed $1$-form $\be_L$ on $K$ such that
\e V=\Ga_p\bigl(t\Phi_L(\gr\be_L)\bigr)
\cup t\Phi_C(\gr\be_C)\bigr)
\cup\Phi_X(\gr\be_X).\e

We take now a Weinstein neighbourhood of $V$ in $(M,\om)$ as follows.
Define a map $f_{LV}:K\to V$ by setting
\[f_{LV}:=\Ga_p\ci t\Phi_L\ci\be_L\]
where $\be_L$ is regarded as a map $K$ into $UL\cap T^*K$.
Likewise we can define two maps $f_{CV}:C\cap A_{tR,\de}\to V$ and $f_{XV}:Z\to V$ by setting
\[f_{CV}:=\Ga_p\ci t\Phi_C\ci\be_C,\,
f_{XV}:=\Phi_X\ci\be_X\]
respectively.
Differentiating $f_{LV}:K\to V$ we get a vector-bundle homomorphism $f_{LV!}:T^*K\to T^*V$ covering $f_{LV}:K\to V$.
Likewise we can also define the two maps $f_{CV!}$ and $f_{XV!}$.
Let $UV$ be a neighbourhood of the zero-section in $T^*V$ such that we can define a map $\Phi_V:UV\to M$ by
\[\Phi_V\ci f_{LV!}=\Ga_p\ci t\Phi_L\ci(+\be_L),\,
\Phi_V\ci f_{CV!}=\Ga_p\ci t\Phi_C\ci(+\be_C)\text{ and }
\Phi_V\ci f_{XV!}=\Phi_L\ci(+\be_X)\]
which defines a Weinstein neighbourhood of $V$ in $(M,\om)$.

By Theorem \ref{zeta gluing} we may write $G\ci F(V)=\Phi_V(\gr\al)$ for some unique closed $1$-form $\al$ on $V$.

\subsubsection*{Second Step}
It remains to prove that $[\al]=0$ in the de Rham cohomology group $H^1(V;\R)$.
The co-efficient field $\R$ will be omitted in what follows for the sake of brevity.
We may topologically write $V$ as a union of $K$ and $Z$ the intersection of which is diffeomorphic to $T^2\t\R$.
We consider the associated Mayer--Vietoris exact sequence.
The map $H^0(K)\op H^0(Z)\to H^0(T^2)$ is clearly surjective, and so the map $H^1(V)\to H^1(K)\op H^1(Z)$ is injective.

It suffices therefore to prove that $[\al]\in H^1(V)$ maps to $0\in H^1(K)\op H^1(Z)$.
We begin by considering its image in $H^1(Z)$.
By the definition of $F$ we have $[\be_X]=t^2Y(L)+[\eta_{F_2(V)}]\in H^1(Z)$ where $\eta_{F_2(V)}$ is as in Lemma \ref{zeta} with $F_2(V)$ in place of $Y$.
On the other hand by the construction of $\Phi_V:UV\to M$ we have
\[\be_X+f_{XV}^*\al=t^2\xi_L+\eta_{F_2(V)}+du_L\text{ on }Z\]
in the notation of Theorem \ref{zeta gluing}.
Consequently $f_{XV}^*[\al]=0\in H^1(Z)$ as we want.

We also prove that $[\al]$ maps to $0$ under the map $H^1(V)\to H^1(K)$, which is induced by $f_{LV}$.
We can naturally compactify $K$ into a manifold with boundary diffeomorphic to $T^2$, which induces a map $H^1(K)\to H^1(T^2)$.
Its kernel is isomorphic to the compactly-supported cohomology group $H^1_c(T^2)$.
Since $K$ is diffeomorphic to $S^1\t\R^2$ it follows that $H^1_c(S^1\t\R^2)\cong H^2(S^1\t\R^2)=\{0\}$
and so the map $H^1(K)\to H^1(T^2)$ is injective.

It suffices therefore to prove that $[\al]\in H^1(V)$ maps to $0\in H^1(T^2)$ under the composite map $H^1(V)\to H^1(K)\to H^1(T^2).$
Recall that it is contained in the Mayer--Vietoris sequence
\[H^1(V)\to H^1(K)\op H^1(Z)\to H^1(T^2)\]
and in particular that the two composite maps $H^1(V)\to H^1(K)\to H^1(T^2)$ and $H^1(V)\to H^1(Z)\to H^1(T^2)$ are equal.
On the other hand we have already proved that $[\al]\in H^1(V)$ maps to $0\in H^1(Z)$ and so to $0\in H^1(T^2)$ as we want, which completes the proof of Lemma \ref{vC}.
\end{proof}

As a corollary to Lemma \ref{vC} we can prove:
\begin{cor}\l{sur}
$G\ci F$ is the identity map of a neighbourhood of $X$ in $\cV$.
\end{cor}
\begin{proof}
It suffices to prove that $V=G\ci F(V)$ in the situation of Lemma \ref{vC}.
Write the exact $1$-form $\al$ as $df$ for some smooth function $f:G\ci F(V)\to\R$.
Since $V$ and $G\ci F(V)$ are special Lagrangian it follows that $f$ satisfies Hopf's maximum principle \cite{Hopf} and so $f$ is constant as $V$ is compact.
Consequently $df=0$ and so $V=G\ci F(V)$ as we want.
\end{proof}

Theorem \ref{main3} is now an immediate consequence of Corollary \ref{sur} and the property (ii) stated before Lemma \ref{vC}.
%Clifford torus cones of dimension greater than $3$
%are not stable by a result of
%Joyce~\cite[p11]{J2} and Marshall~\cite[\S6.3.4]{Marshall}.
%There seems to be few known examples of stable
%special Lagrangian cones except the Clifford torus cone of dimension $3$.
%For instance Ohnita~\cite{Ohnita} gives an example of stable cone on $SU_3/SO_3$,
%but Haskins and Pacini~\cite[Remark~2.3]{Haskins Pacini}
%point out that $SU_3/SO_3$ does not bound any compact manifold
%because $SU_3/SO_3$ has a non-zero Stiefel--Whitney number;
%this implies that we cannot smooth the singularity modelled on
%the cone of $SU_3/SO_3$, which is not suitable for our purpose.

%We note that in Theorem~\ref{main3} we suppose $X\in\d\mathcal{N}$
%but we have also:
%\begin{cor*}[to Corollary~\ref{gluing condition}]
%If $X\notin\d\mathcal{N}$ then there exists a neighbourhood $\cU$
%of $X$ in $\cV$ with $\cU\subset\mathcal{X}$.
%\end{cor*}
%\begin{proof}
%We give a proof by contradiction,
%and so suppose that for every neighbourhood $\cU$
%of $X$ in $\cV$ with $\cU\-\mathcal{X}\neq\emptyset$.
%By Corollary~\ref{gluing condition}, then, we shall have
%$L\in\{L_1,L_2,L_3\}$ with $Y(L)\in\im f_X^*$,
%and so we can glue $L$ to $X$ by the method of Joyce.
%This contradicts $X\notin\d\mathcal{N}$
%and so completes the proof.
%\end{proof}

\end{document}